\documentclass[a4paper,11pt,fullpage]{article}
\usepackage[english]{babel}

\usepackage[left=1in,right=1in,top=1in,bottom=1in]{geometry}

\usepackage{amsmath}
\usepackage{amsthm}
\usepackage{amsfonts}
\usepackage{amssymb}
\usepackage{amsxtra}
\usepackage{latexsym}
\usepackage{ifthen}
\usepackage{cite}
\usepackage{mathtools}
\usepackage{esint}
\usepackage[cp1250]{inputenc}

\usepackage{epic}
\usepackage{eepic}

\DeclareMathOperator*{\esssup}{ess\,sup}

\DeclareMathOperator*{\essinf}{ess\,inf}

\numberwithin{equation}{section}
\newtheorem{theorem}{Theorem}[section]
\newtheorem{lemma}[theorem]{Lemma}
\newtheorem{remark}[theorem]{Remark}
\newtheorem{definition}[theorem]{Definition}

\def\XXint#1#2#3{{\setbox0=\hbox{$#1{#2#3}{\int}$}
     \vcenter{\hbox{$#2#3$}}\kern-.5\wd0}}

\begin{document}

%\Large

\title{Continuity at a boundary point of solutions to quasilinear elliptic equations
with generalized Orlicz growth and non-logarithmic conditions}

\author{Oleksandr V. Hadzhy,  \ \  Mykhailo V. Voitovych}
%\email{iskrypnik@iamm.donbass.com}
%\thanks{The Division of Applied Problems
% in Contemporary Analysis, Institute of Mathematics of NASU, Kiev, Ukraine}

\maketitle

\begin{abstract}
We consider the Dirichlet problem for quasilinear elliptic equations
with Musielak-Orlicz $(p,q)$-growth and non-logarithmic conditions
on the coefficients.
A sufficient Wiener-type condition for the regularity of a boundary point is established.

\textbf{Keywords:} nonlinear elliptic equations; Musielak-Orlicz growth;
non-logarithmic condition;
%Dirichlet problem;
boundary regularity; Wiener condition.

\textbf{MSC (2020)}: 35B65, 35D30, 35J60.

\end{abstract}

\pagestyle{myheadings} \thispagestyle{plain} \markboth{Oleksandr V. Hadzhy, \ \ Mykhailo V. Voitovych}
{Continuity at a boundary point of solutions to quasilinear elliptic equations
. . .}

\section{Introduction}\label{Introduction}

Let $\Omega$ be a bounded domain in $\mathbb{R}^{n}$ ($n\geqslant2$),
$\overline{\Omega}$ is the closure of $\Omega$ in $\mathbb{R}^{n}$,
and $\partial\Omega=\overline{\Omega}\setminus \Omega$ is the boundary of $\Omega$.
In this paper we study the regularity of boundary points ($x_{0}\in \partial\Omega$)
for bounded solutions to the quasilinear elliptic equations with Musielak-Orlicz
$(p,q)$-growth and non-logarithmic Zhikov's condition on the coefficients.
In more detail this means that we consider equations of the form
\begin{equation}\label{gellequation}
{\rm div}\bigg( g(x,|\nabla u|)\frac{\nabla u}{|\nabla u|} \bigg)=0,
\quad x\in \Omega,
\end{equation}
where the function
%
%
%Throughout the paper we suppose that the function
$g(x, {\rm v}): \mathbb{R}^{n}\times \mathbb{R}_{+}\rightarrow \mathbb{R}_{+}$,
$\mathbb{R}_{+}:=[0,+\infty)$
satisfies the following assumptions:
\begin{itemize}
\item[(${\rm g}_{0}$)]
$g(\cdot, {\rm v})\in L^{1}(\Omega)$ for all ${\rm v}\in \mathbb{R}_{+}$,
$g(x, \cdot)$ is continuous and increasing for all $x\in \mathbb{R}^{n}$,
$\lim\limits_{{\rm v}\rightarrow +0}g(x, {\rm v})=0$ and
$\lim\limits_{{\rm v}\rightarrow +\infty}g(x, {\rm v})=+\infty$;
%$$
%\lim\limits_{{\rm v}\rightarrow +0}g(x, {\rm v})=0 \ \ \text{and} \ \
%\lim\limits_{{\rm v}\rightarrow +\infty}g(x, {\rm v})=+\infty.
%$$
\end{itemize}
\begin{itemize}
\item[(${\rm g}_{1}$)]
there exist $c_{1}>0$, $q>1$ and $b_{0}>0$ such that
\begin{equation}\label{gqineq}
\frac{g(x, {\rm w})}{g(x, {\rm v})}\leqslant
c_{1} \left( \frac{{\rm w}}{{\rm v}} \right)^{q-1},
\end{equation}
for all $x\in \overline{\Omega}$ and for all ${\rm w}\geqslant{\rm v}\geqslant b_{0}$;
\end{itemize}
\begin{itemize}
\item[(${\rm g}_{2}$)]
there exists $p>1$ such that
\begin{equation}\label{gpineq}
\frac{g(x, {\rm w})}{g(x, {\rm v})}\geqslant
\left( \frac{{\rm w}}{{\rm v}} \right)^{p-1},
\end{equation}
for all $x\in \overline{\Omega}$ and for all ${\rm w}\geqslant{\rm v}>0$;
\end{itemize}
\begin{itemize}
\item[(${\rm g}_{3}$)]
there exist a non-decreasing function $c_{2}: \mathbb{R}_{+}\rightarrow \mathbb{R}_{+}$
and a continuous, non-increasing function $\lambda:(0, r_{\ast})\rightarrow \mathbb{R}_{+}$
such that for any ball $B_{r}(x_{0})$ centered at $x_{0}\in \overline{\Omega}$,
%$B_{r}(x_{0}):=\{x\in \mathbb{R}^{n}: |x-x_{0}|<r\}$,
for all
$x_{1}$, $x_{2}\in B_{r}(x_{0})$
and for all $0<r\leqslant {\rm v}\leqslant K$, the following inequality holds:
\begin{equation}\label{eqnonlogcond}
g(x_{1},  {\rm v}/r)\leqslant c_{2}(K)\,e^{\lambda(r)} g(x_{2},  {\rm v}/r).
\end{equation}
\end{itemize}
These assumptions are quite general and cover a wide class of elliptic equations with non-standard
growth conditions. %for the coefficients.
For example, the following functions: %satisfy assumptions
%(${\rm g}_{0}$)--(${\rm g}_{3}$):
\begin{itemize}
\item[1)] variable exponent $g(x, {\rm v})= {\rm v}^{\,p(x)-1}$,
cf. \cite{Alhutov97, Fan1995, FanZhao1999},
\item[2)] perturbed variable exponent
$g(x, {\rm v})= {\rm v}^{\,p(x)-1}\ln(e+{\rm v})$, cf.
\cite{GianPasdiNap2013, LiangCaiZheng, Ok_CalcVar2016, Ok_AdvNonlinAn2018},
\item[3)] double phase $g(x, {\rm v})= {\rm v}^{\,p-1}+a(x){\rm v}^{\,q-1}$,
$0\leqslant a(x)\in L^{\infty}(\Omega)$, cf.
\cite{BarColMing, BarColMingStPt16, BarColMingCalc.Var.18, ColMing218, ColMing15},
\item[4)] degenerate double phase
$g(x, {\rm v})={\rm v}^{\,p-1} \big(1+b(x)\ln(1+{\rm v}) \big)$,
$0\leqslant b(x)\in L^{\infty}(\Omega)$, cf.
\cite{BalciSurn, BarColMingStPt16, ByunOh},
\item[5)]
double variable exponent
$g(x, {\rm v})= {\rm v}^{\,p(x)-1}+{\rm v}^{\,q(x)-1}$,
cf. \cite{CenRadRep_NA2018, ShiRadRepZh_AdvCalcVar2020, ZhangRad_JMathPAppl2018},
\item[6)] variable exponent double phase
$g(x, {\rm v})= {\rm v}^{\,p(x)-1}+a(x){\rm v}^{\,q(x)-1}$,
$0\leqslant a(x)\in L^{\infty}(\Omega)$,
cf. \cite{MaedMizOhnShim2019, RagTach2020}
\end{itemize}
%
%$$
%g(x, {\rm v}):= {\rm v}^{\,p(x)-1} \quad (\text{variable exponent}),
%$$
%\begin{equation}\label{examplfnctg12}
%g(x, {\rm v}):= {\rm v}^{\,p(x)-1}+{\rm v}^{\,q(x)-1},
%\quad
%g(x, {\rm v}):={\rm v}^{\,p(x)-1} \big(1+\ln(1+{\rm v}) \big),
%\end{equation}
%\begin{equation}\label{examplfnctg34}
%g(x, {\rm v}):= {\rm v}^{\,p-1}+a(x){\rm v}^{\,q-1}, \
%a(x)\geqslant 0, \quad
%g(x, {\rm v}):={\rm v}^{\,p-1} \big(1+b(x)\ln(1+{\rm v}) \big), \
%b(x)\geqslant 0,
%\end{equation}
%\begin{equation}\label{examplfnctg34}
%\begin{aligned}
%&g(x, {\rm v}):= {\rm v}^{\,p-1}+a(x){\rm v}^{\,q-1},
%&a(x)\geqslant 0 \quad (\text{double phase}), \\
%&g(x, {\rm v}):={\rm v}^{\,p-1} \big(1+b(x)\ln(1+{\rm v}) \big),
%&b(x)\geqslant 0,
%\end{aligned}
%\end{equation}
satisfy assumptions
(${\rm g}_{0}$)--(${\rm g}_{3}$) if the exponents $p$, $q$, $p(\cdot)$, $q(\cdot)$
and the coefficients $a(\cdot)$
and $b(\cdot)$ are such that:
%satisfy the  following conditions:
%$1<p< p(x)\leqslant q(x)< q<+\infty$ for all $x\in \Omega$,
$$
\text{(i)} \quad
1<p< p(x)\leqslant q(x)< q<+\infty \quad \text{for all } \ x\in \Omega;
\hskip 28mm
$$
\begin{equation}\label{eqSkr1.3}
\begin{aligned}
\text{(ii)} \quad &|p(x)-p(y)|+|q(x)-q(y)|\leqslant
\frac{\lambda(|x-y|)}
{\big|\ln |x-y|\big|},
\quad x,y\in \Omega, \quad x\neq y, \
\\
& \text{the function} \ \frac{\lambda(r)}{|\ln r|}
\text{ is non-decreasing on } (0,r_{\ast}), \
\lim\limits_{r\rightarrow 0} \frac{\lambda(r)}{|\ln r|}=0;
\hskip 5 mm
\end{aligned}
\end{equation}
\begin{equation}\label{axconditin}
\begin{aligned}
\text{(iii)} \quad &|a(x)-a(y)|\leqslant A|x-y|^{\alpha} e^{\lambda(|x-y|)},
\quad x,y\in \Omega,
\quad x\neq y,
\\
&A>0, \quad 0<q-p\leqslant\alpha \leqslant 1,
\\
& \text{the function} \ r^{\alpha}e^{\lambda(r)} \text{ is non-decreasing on } (0,r_{\ast}), \
\lim\limits_{r\rightarrow 0}r^{\alpha}\,e^{\lambda(r)}=0;
\end{aligned}
\end{equation}
\begin{equation}\label{bxconditin}
\begin{aligned}
\text{(iv)} \quad &|b(x)-b(y)|\leqslant  \frac{B\,e^{\lambda(|x-y|)}}{\big|\ln |x-y| \big|},
\quad x,y\in \Omega, \quad x\neq y, \quad B>0,
\\
& \text{the function} \ \frac{e^{\lambda(r)}}{|\ln r|}\text{ is non-decreasing on } (0,r_{\ast}),
\ \lim\limits_{r\rightarrow 0} \frac{e^{\lambda(r)}}{|\ln r|}=0. \hskip 5 mm
\end{aligned}
\end{equation}

Recently, many studies have been devoted to questions of regularity under Zhikov's
\cite{ZhikMathPhys95} and Fan's \cite{Fan1995}  logarithmic
condition, when  $\lambda(r)\leqslant L<+\infty$ in \eqref{eqSkr1.3}--\eqref{bxconditin}
(see e.g.
%Alkhutov \cite{Alhutov97}, Balci \& Surnachev \cite{BalciSurn},
%Baroni, Colombo \& Mingione \cite{BarColMing, BarColMingStPt16, BarColMingCalc.Var.18},
%Byun \& Oh \cite{ByunOh}, Colombo \& Mingione \cite{ColMing218, ColMing15},
%Fan \& Zhao \cite{FanZhao1999}, Giannetti \& Passarelli di Napoli \cite{GianPasdiNap2013},
%Liang, Cai, \& Zheng\cite{LiangCaiZheng}, Ok \cite{Ok_CalcVar2016, Ok_AdvNonlinAn2018},
%Ragusa \& Tachikawa \cite{RagTach2020}, Voitovych \cite{VoitNA19}),
\cite{Alhutov97, BalciSurn, BarColMing, BarColMingStPt16,
BarColMingCalc.Var.18, ByunOh, ColMing218, ColMing15, FanZhao1999,
GianPasdiNap2013, LiangCaiZheng, Ok_CalcVar2016, Ok_AdvNonlinAn2018, RagTach2020, VoitNA19}),
and now the elliptic theory has reached a completely satisfactory form
(see original surveys
%of
%Harjulehto, H\"{a}st\"{o}, L\^{e} \& Nuortio \cite{HarHastLeNuorNA2010},
%Mingione \cite{MingioneDarkSide}, Mingione \& R\u{a}dulescu \cite{MingRadul},
%R\u{a}dulescu \cite{Radul_NA2015} and monograph of R\u{a}dulescu \& Repov\v{s} \cite{RadRep}.
\cite{HarHastLeNuorNA2010, MingioneDarkSide, MingRadul, Radul_NA2015}
and monograph \cite{RadRep}).
In general, the logarithmic condition has significantly advanced the theory of function spaces
with variable exponents \cite{CruzUrFior, DienHarHastRuzVarEpn},
which are an integral part of the generalized
Orlicz (Musielak-Orlicz) spaces \cite{HarHastOrlicz, Musielak}.
Nonlinear differential equations in Muselak-Orlicz spaces is a fairly extensive topic
for research (see the informative survey \cite{ChlebPock}).
Here we only cite the recent papers
\cite{BenHarHastKarp, BenKhlifi, HarHastLee, HastOk}
%of
%Benyaiche, Harjulehto, H\"{a}st\"{o} \& Karppinen \cite{BenHarHastKarp},
%Benyaiche \& Khlifi \cite{BenKhlifi},
%Harjulehto, H\"{a}st\"{o} \& Lee \cite{HarHastLee},
%H\"{a}st\"{o} \& Ok \cite{HastOk}
on H\"{o}lder's regularity and  Harnack's inequality under the
Musielak-Orlicz assumptions and logarithmic conditions in the elliptic framework.

The non-logarithmic case, when the function $\lambda(r)$ is unbounded in \eqref{eqnonlogcond},
\eqref{eqSkr1.3}--\eqref{bxconditin}, differs significantly from the logarithmic one.
To our knowledge there are few regularity results in this direction. Zhikov \cite{ZhikPOMI04}
obtained a generalization of the logarithmic condition on the variable exponent
$p(x)\geqslant p>1$ which guarantees the density
of smooth functions in the Sobolev space $W^{1,p(x)}(\Omega)$
: %$|p(x)-p(y)|\leqslant L\left|\ln \big|\ln |x-y|\big| \right|/\big|\ln |x-y|\big|$
\begin{equation}\label{Zhik2log}
|p(x)-p(y)|\leqslant L\,
\frac{ \left|\ln \big|\ln |x-y|\big| \right|}{\big|\ln |x-y|\big|},
\quad x,y\in\Omega, \quad x\neq y, \quad L<p/n.
\end{equation}
Later Zhikov and Pastukhova \cite{ZhikPast2008MatSb}
under the same condition proved higher integrability
of the gradient of solutions to the $p(x)$-Laplace equation.

%In the case when the variable exponent $p(x)$ satisfies the condition
%\begin{equation}\label{lnlnlnZhikcond}
%|p(x)-p(x_{0})|\leqslant L\,\frac{\ln\ln\ln|x-x_{0}|^{-1}}{\ln|x-x_{0}|^{-1}},
%\ \ 0<L<p/(n+1), \ \ x,x_{0}\in\Omega,  \ \ 0<|x-x_{0}|<1/27,
%\end{equation}
Interior continuity, continuity up to the boundary and Harnack's inequality to
$p(x)$-Laplace equation were proved by
Alkhutov and Krasheninnikova \cite{AlhutovKrash08},
Alkhutov and Surnachev \cite{AlkhSurnAlgAn19} and Surnachev \cite{SurnPrepr2018}
under condition \eqref{eqSkr1.3} on $p(\cdot)$ and the additional assumption
\begin{equation}\label{AlkSurncond}
\int_{0}\exp\Big(-C\exp\big( \beta\lambda(r) \big)  \Big)\,
\frac{dr}{r}=+\infty
\end{equation}
with some constants $C>0$, $\beta>1$ depending only upon the data.
Particularly, the function $\lambda(r)=L\ln\ln\ln r^{-1}$, $L\beta<1$ %$r\in(0,e^{-e})$,
satisfies the above conditions.
%condition \eqref{AlkSurncond}.
These results were generalized in \cite{SkrVoitarXiv20, ShSkrVoit}
for a wide class of elliptic and parabolic equations with non-logarithmic
generalized Orlicz growth, and improved in \cite{HadzhySkrVoit} for double phase elliptic
equations (see type 3 and \eqref{axconditin})
with $\lambda(r)=L\ln\ln r^{-1}$ (cf. \cite{ZhikPOMI04, ZhikPast2008MatSb}
and \eqref{Zhik2log}).

%proved the continuity of solutions to the $p(x)$-Laplace
%equation at the point $x_{0}$, and Surnachev \cite{SurnPrepr2018}
%established the interior Harnack inequality for solutions.
%The continuity of solutions to the $p(x)$-Laplace equation up to the boundary  were proved
%by Alkhutov and Surnachev \cite{AlkhSurnAlgAn19} under condition \eqref{eqSkr1.3} on $p(\cdot)$
%and the additional assumption

The purpose of this article is to establish a Wiener-type sufficient condition for the regularity
of a boundary point $x_{0}\in \partial\Omega$ for bounded solutions of Eq.
\eqref{gellequation} under assumptions (${\rm g}_{0}$)--(${\rm g}_{3}$).
Boundary regularity in terms of the Sobolev variational capacity is a classical topic in the
contemporary theory of PDE, which goes back to
Lebesgue \cite{Lebesgue} and Wiener \cite{Wienern, WienerDP}.
The  papers
\cite{LitStWei}, \cite{Mazya1970}, \cite{GarZiem}, \cite{Skryp1984},
\cite{LindMart}, \cite{KilpMaly}
%of Littman, Stampacchia \& Weinberger \cite{LitStWei},
%Maz'ya \cite{Mazya1970}, Gariepy \& Ziemer \cite{GarZiem},
%Skrypnik \cite{Skryp1984}, Lindqvist \& Martio \cite{LindMart},
%Kilpel\"{a}inen \& J. Mal\'{y} \cite{KilpMaly}
provided the basic contribution to the boundary regularity for
quasilinear elliptic equations with standard growth conditions.
For the variable $p(x)$-growth, we cite \cite{AlhKrash04} in the log-case,
and \cite{AlkhSurnAlgAn19} for non-logarithmic conditions
\eqref{eqSkr1.3}, \eqref{AlkSurncond} on the exponent $p(\cdot)$.
In the Orlicz case, the Wiener criterion for the regularity of Sobolev boundary point
was established by Ki-Ahm Lee and Se-Chan Lee \cite{KiALeeSe-Ch_JDE2021}.
%Very recently, Ki-Ahm Lee and Se-Chan Lee \cite{KiALeeSe-Ch_JDE2021}
%proved the Wiener criterion for the regularity
%of Sobolev boundary point in the Orlicz case.
Under generalized Orlicz growth and logarithmic counterparts of \eqref{eqnonlogcond},
Harjulehto and H\"{a}st\"{o} \cite{HarHastboundreg} obtained the following capacity density
condition for the regularity
of a boundary point $x_{0}\in\partial\Omega$: there exists $c\in (0,1)$ and $R>0$ such that
$
\mathrm{Cap}_{G}\big(B_{\rho}(x_{0})\setminus\Omega, B_{2\rho}(x_{0})\big)
\geqslant c\,
\mathrm{Cap}_{G}\big(B_{\rho}(x_{0}), B_{2\rho}(x_{0})\big)$
for all $0<\rho<R$, $G(\cdot, {\rm v})\asymp g(\cdot, {\rm v}){\rm v}$.
This result was improved by Benyaiche and Khlifi \cite{BenKhlifiWiener} to a full-fledged
Wiener-type criterion: the point $x_{0}$ is regular if and only if for some $R>0$,
$$
\int_{0}^{R}g^{-1}_{x_{0}}
\bigg( \frac{{\rm Cap}_{G}\big(B_{\rho}(x_{0})
\setminus\Omega, B_{2\rho}(x_{0}) \big)}{\rho^{n-1}}
 \bigg)\,d\rho=+\infty
$$
where $g^{-1}_{x_{0}}(\cdot)$ denotes the inverse function to the function $g(x_{0}, \cdot)$.

In this paper, we show (see Theorem \ref{thregpoint}) that the corresponding regularity
condition for bounded solutions to Eq. \eqref{gellequation}
under assumptions (${\rm g}_{0}$)--(${\rm g}_{3}$) has the following form:
$$
\int_{0}g^{-1}_{x_{0}}
\bigg(\Lambda(-C,3n,\rho)\,
\frac{{\rm Cap}_{G}\big(\overline{B}_{\rho/32}(x_{0})
\setminus\Omega, B_{\rho}(x_{0}) \big)}{\rho^{n-1}}
 \bigg)\,d\rho=+\infty,
$$
where $C$ is a positive constant depending only upon the data, and
\begin{equation}\label{defLambda}
\Lambda(c,\beta,\rho)=\exp\Big(c\exp\big( \beta\lambda(\rho) \big)  \Big)
\quad
\text{for any} \ c, \beta\in \mathbb{R} \ \text{and} \ \rho\in(0,r_{\ast}).
\end{equation}
Thus, we expand the previous results in several directions.
Namely, the result of Alkhutov and Surnachev \cite{AlkhSurnAlgAn19} extends from the
$p(x)$-Laplace equation to equations with generalized Orlicz growth,
and the results of Benyaiche and Khlifi \cite{BenKhlifi, BenKhlifiWiener} extend
to the non-logarithmic generalized Orlicz growth.
In particular, our results are new for bounded solutions to equations of types 2, 4--6.
The main results are set out in Section \ref{Sectmnres}, and
their proofs are contained in Sections \ref{proofregcond}--\ref{Sectgrest}.

\section{Main Results}\label{Sectmnres}

Before formulating the main results, we remind the definition of a weak solution
to Eq. \eqref{gellequation}.
Moreover, throughout the article, we use the well-known notation for sets in $\mathbb{R}^{n}$,
for function spaces and their elements etc. (see, for instance,
\cite{GilbTr, HarHastOrlicz, LadUr}).
We set
\begin{equation}\label{deffuncgw}
G(x, {\rm v})=g(x,{\rm v}){\rm v} \ \ \text{for} \ \ x\in\mathbb{R}^{n}, \ {\rm v}\geqslant0
\end{equation}
and write $u\in W^{1,G}(\Omega)$ if $u\in W^{1,1}(\Omega)$
and $\int_{\Omega}G(x, |\nabla u|)\,dx<+\infty$.
%$u\in W^{1,G}_{{\rm loc}}(\Omega)$ if $u\in W^{1,G}(E)$
%for any open set $E$ compactly embedding in $\Omega$.
We also need a class of functions
$W^{1,G}_{0}(\Omega):=W_{0}^{1,1}(\Omega)\cap W^{1,G}(\Omega)$.
%We denote by $W^{1,G}_{0}(\Omega)$ the set of all functions
%$u\in W^{1,G}(\Omega)$ which have a compact support in $\Omega$.

\begin{definition}[definition of solutions]\label{defsolution}
{\rm We say that a function $u:\Omega\rightarrow\mathbb{R}$ is a bounded weak solution
%(subsolution, supersolution)
to Eq. \eqref{gellequation} under hypotheses (${\rm g}_{0}$)--(${\rm g}_{3}$) if
$u\in W^{1,G}(\Omega)\cap L^{\infty}(\Omega)$ and for any $\varphi\in W^{1,G}_{0}(\Omega)$
the following integral identity holds: %(inequality)
\begin{equation}\label{gelintidentity}
\int_{\Omega}
g(x,|\nabla u|)\,\frac{\nabla u}{|\nabla u|}\,\nabla\varphi\,dx=0. %(\leqslant\,,\geqslant)\,0
\end{equation}
%(for subsolutions and supersolutions, we require $\varphi\geqslant 0$).
}
\end{definition}

Further we also need the following definitions.
\begin{definition}[regular boundary points]\label{defregpoint}
{\rm
We say that $x_{0}\in \partial\Omega$ is a regular boundary point of the domain $\Omega$
for Eq. \eqref{gellequation} if for every $f\in C(\overline{\Omega})\cap W^{1,G}(\Omega)$
and for every bounded weak solution $u$ of Eq. \eqref{gellequation} satisfying the condition
$u-f\in W^{1,G}_{0}(\Omega)$ the following equality holds
$\lim\limits_{\Omega\ni x\rightarrow x_{0}}u(x)=f(x_{0})$.
}
\end{definition}
%%%%%%%%%%%%%%%%%%%%%%%%%%%%%%%%%%%%%%%%%%%%%%%%%%%%%%%%%%%%%%%%%%%%%%%%%%%%%%%%%%%%%%%%%%%%%%%%%%%%%%%%%%%%%%
\begin{definition}[capacity]\label{defcapacity}
{\rm Let $E$ be a compact subset of $B_{\rho}(x_{0})$,
and let $\mathfrak{M}=\mathfrak{M}(E,B_{\rho}(x_{0}))$ be
the class  of all functions
$v\in W_{0}^{1,G}(B_{\rho}(x_{0}))$ satisfying $v\geqslant1$ on $E$
in the sense of $W_{0}^{1,1}(B_{\rho}(x_{0}))$.
Then the relative $G$-capacity of $E$
is defined by
%by definition for every set $E\subset B_{\rho/32}(x_{0})$, %and every $L>0$
\begin{equation}\label{defCap}
\mathrm{Cap}_{G}\big(E, B_{\rho}(x_{0})\big)=
\inf\limits_{v\in\mathfrak{M}}
\int_{B_{\rho}(x_{0})} G(x, |\nabla \varphi|)\,dx.
\end{equation}
}
\end{definition}

Now we can  state our main result.
In this case, we refer to the parameters
$M=\esssup_{\Omega} |u|$, $n$, $p$, $q$, $c_{1}$, $c_{2}(M)$, $b_{0}$
as our structural data and write $\gamma$ for constants if they can be quantitatively
determined a priori only in terms of the above quantities.
The generic constant $\gamma$ may vary from line to line.
%We also set

\begin{theorem}[Wiener type regularity condition]\label{thregpoint}
%Let assumptions $({\rm g})$, $({\rm g}_{1})$--$({\rm g}_{3})$ be fulfilled.
Let hypotheses $({\rm g}_{0})$--$({\rm g}_{3})$ be fulfilled.
Then a boundary point $x_{0}\in \partial\Omega$ is regular in the sense of
Definition \ref{defregpoint} if under notation \eqref{defLambda}, \eqref{deffuncgw}
and \eqref{defCap} the following condition is satisfied:
\begin{equation}\label{Winercond}
\int_{0}g^{-1}_{x_{0}}
\bigg(\Lambda(-\gamma,3n,\rho)\,
\frac{{\rm Cap}_{G}\big(\overline{B}_{\rho/32}(x_{0})
\setminus\Omega, B_{\rho}(x_{0}) \big)}{\rho^{n-1}}
 \bigg)\,d\rho=+\infty.
\end{equation}
Here we use the notation $g^{-1}_{x_{0}}(\cdot)$
for the inverse function to the function $g(x_{0}, \cdot)$.
\end{theorem}

The proof of Theorem \ref{thregpoint} is carried out by contradiction
(see Section \ref{proofregcond}) in the spirit of the classic paper by
R. Gariepy and W.\,P. Ziemer \cite{GarZiem}.
In fact, Theorem \ref{thregpoint} is a consequence of an integral growth estimate
on the gradient of solutions (cf. \cite[Theorem\,2.1]{GarZiem}).
To state it we need additional notations.
Let $f\in C(\overline{\Omega})\cap W^{1,G}(\Omega)$ and
$u$ be a bounded weak solution of Eq. \eqref{gellequation} satisfying the condition
$u-f\in W^{1,G}_{0}(\Omega)$. Next, let $x_{0}\in \partial\Omega$ be a boundary point,
and let $l\in \mathbb{R}$, $l\neq f(x_{0})$. By the continuity of $f$ on $\partial\Omega$,
there exists a sufficiently small radius $\rho>0$ such that
$$
l\geqslant M_{f}(r):=\sup\limits_{\partial\Omega\cap B_{r}(x_{0})}f
\quad \text{if } \ l>f(x_{0}) \ \text{ and } \ 0<r\leqslant \rho,
$$
$$
l\leqslant m_{f}(r):=\inf\limits_{\partial\Omega\cap B_{r}(x_{0})}f
\quad \text{if } \ l<f(x_{0}) \ \text{ and } \ 0<r\leqslant \rho.
$$
We set $c_{+}=\max\{c,0\}$ for every $c\in\mathbb{R}$,
\begin{equation}\label{defOmegalr}
\Omega_{l,r}=\begin{cases}
\{u>l\}\cap B_{r}(x_{0})
& \text{if } \ l>f(x_{0}) \ \text{and} \ 0<r\leqslant \rho, \\
\{u<l\}\cap B_{r}(x_{0})
& \text{if } \ l<f(x_{0}) \ \text{and} \ 0<r\leqslant \rho,
\end{cases}
%\hskip 119pt
\end{equation}
\begin{equation}\label{defulrho}
u_{l}=\begin{cases}
(u-l)_{+}
& \text{on } \Omega \ \text{ if } \ l\geqslant M_{f}(\rho), \\
(l-u)_{+}
& \text{on } \Omega \ \text{ if } \ l\leqslant m_{f}(\rho), \\
0
& \text{on } \mathbb{R}^{n}\setminus\Omega,
\end{cases}
\quad\quad
\begin{aligned}
M_{l}(r)&=\esssup\limits_{B_{r}(x_{0})}u_{l}
\ \text{ for} \ 0<r\leqslant \rho,
\\
u_{l,\rho}&=M_{l}(\rho)-u_{l}+2b_{0}\rho.
\end{aligned}
\end{equation}

\begin{theorem}[growth gradient estimate]\label{thgradest}
Let hypotheses $({\rm g}_{0})$--$({\rm g}_{3})$ be fulfilled.
Then, in terms of notation \eqref{defLambda}, \eqref{deffuncgw},  \eqref{defOmegalr}
and \eqref{defulrho}, the following inequality holds:
\begin{equation}\label{eq1.7}
\begin{aligned}
&\int_{\Omega_{l,\rho/16}}G(x, |\nabla(u_{l,\rho}\eta)|)\,dx
\\
&\leqslant \Lambda(\gamma, 3n,\rho)(M_{l}(\rho)+2b_{0}\rho)\rho^{n-1}
g\bigg( x_{0},\frac{M_{l}(\rho)-M_{l}(\rho/4)+2b_{0}\rho}{\rho} \bigg).
\end{aligned}
\end{equation}
Here %$\Omega_{l,\rho/16}=\{u>l\}\cap B_{\rho/16}(x_{0})$, and
 $\eta\in C_{0}^{\infty}(B_{\rho/16}(x_{0}))$ is a function such that
$0\leqslant\eta\leqslant1$, $\eta=1$ in $B_{3\rho/64}(x_{0})$,
$|\nabla\eta|\leqslant 64/\rho$.
%$\Lambda(c,\beta,\rho)= \exp\Big(c\exp \big( \beta\lambda(\rho)  \big)  \Big)$
%for any $c$, $\beta\in \mathbb{R}$ and $\rho\in (0,r_{\ast})$.
\end{theorem}

In turn, the proof of Theorem \ref{thgradest} essentially relies on the weak Harnack inequality
for solutions of Eq. \eqref{gellequation} at the boundary $\partial\Omega$.
Let us state the corresponding result.
In this case, we also use the notation
\begin{equation}\label{defintsrf}
\fint_{E}f\,dx:=|E|^{-1}\int_{E}f\,dx
\end{equation}
for any measurable set $E\subset \mathbb{R}^{n}$ with $|E|\neq 0$ and $f\in L^{1}(E)$,
where $|E|$ denotes the $n$-dimensional Lebesgue measure of $E$.
\begin{theorem}[boundary weak Harnack inequality]\label{thweakHarnineq}
Let hypotheses $({\rm g}_{0})$--$({\rm g}_{3})$ be fulfilled.
Then, for every $0<s<n/(n-1)$, the following inequality holds in terms of notation
\eqref{defLambda}, \eqref{defulrho} and \eqref{defintsrf}:
\begin{equation}\label{weakHrnckineq}
\bigg(\fint_{B_{\rho/8}(x_{0})}
g^{s}( x_{0},u_{l,\rho}/\rho )\,dx\bigg)^{1/s}
\leqslant\Lambda(\gamma,3n,\rho)\,
g\bigg( x_{0},\frac{M_{l}(\rho)-M_{l}(\rho/4)+2b_{0}\rho}{\rho} \bigg).
\end{equation}
Here a positive constant $\gamma$ additionally depends on $s$.
%$M=\esssup_{\Omega} |u|$, $n$, $p$, $q$, $c_{1}$, $c_{2}(M)$, $b_{0}$.
%where $\Lambda(c,\beta,\rho)= \exp\Big(c\exp \big( \beta\lambda(\rho)  \big)  \Big)$
%for any $c$, $\beta\in \mathbb{R}$ and $\rho\in (0,r_{\ast})$.
\end{theorem}
\begin{remark}\label{remHarnineq}
{\rm The interior weak and strong Harnack's inequalities for bounded solutions of
Eq. \eqref{gellequation} under hypotheses $({\rm g}_{0})$--$({\rm g}_{3})$
were established in \cite{ShSkrVoit} (see
also \cite{BenKhlifi, BenHarHastKarp} in the context of
generalized Orlicz growth and logarithmic condition).
But unlike \cite{BenKhlifi, BenHarHastKarp}, the traditional Moser method
\cite{Moser1960, Moser1961} is inapplicable for the proof of Theorem \ref{thweakHarnineq},
since checking whether
the logarithm of solutions belongs to the BMO space is a problem for the non-logarithmic condition
\eqref{eqnonlogcond} with unbounded $\lambda(r)$.
We use Trudinger's  arguments \cite{TrudingerArch71} adapted to Eq. \eqref{gellequation}
near the boundary $\partial\Omega$ (see Section \ref{SectHarnproof}).
%for unbounded supersolutions
%and \cite{BenKhlifi} for bounded non-negative solutions  under
%(see \cite{ShSkrVoit} for the interior inequality)
}
\end{remark}

%%%%%%%%%%%%%%%%%%%%%%%%%%%%%%%%%%%%%%%%%%%%%%%%%%%%%%%%%%%%%%%%%%%%%%%%%%%%%%%%%%%%%%%%%%%%%%%%%%%%%%%%%%%%%%%%%%%%%%%
%%%%%%%%%%%%%%%%%%%%%%%%%%%%%%%%%%%%%%%%%%%%%%%%%%%%%%%%%%%%%%%%%%%%%%%%%%%%%%%%%%%%%%%%%%%%%%%%%%%%%%%%%%%%%%%%%%%%%%%
%%%%%%%%%%%%%%%%%%%%%%%%%%%%%%%%%%%%%%%%%%%%%%%%%%%%%%%%%%%%%%%%%%%%%%%%%%%%%%%%%%%%%%%%%%%%%%%%%%%%%%%%%%%%%%%%%%%%%%%%%

\section{A sufficient condition for the regularity of a boundary point:
 proof of Theorem \ref{thregpoint}}\label{proofregcond}

Let  $x_{0}\in \partial\Omega$, $f\in C(\overline{\Omega})\cap W^{1,G}(\Omega)$,
and let $u$ be a bounded weak solution  of Eq. \eqref{gellequation} satisfying the condition
$u-f\in W^{1,G}_{0}(\Omega)$.
We need to prove that $\lim\limits_{\Omega\ni x\rightarrow x_{0}}u=f(x_{0})$.
This equality will be established if we show that
$f(x_{0})\leqslant \lim\limits_{\rho\rightarrow0}\essinf\limits_{\Omega\cap B_{\rho}(x_{0})}u$
and
$\lim\limits_{\rho\rightarrow0}\esssup\limits_{\Omega\cap B_{\rho}(x_{0})} u\leqslant f(x_{0})$.
The proves of the both inequalities are completely similar and we will prove only the second one.
We argue by contradiction and assume that
\begin{equation}\label{Lfhyp}
L:=\lim\limits_{\rho\rightarrow0}\esssup\limits_{\Omega\cap B_{\rho}(x_{0})} u>f(x_{0}).
\end{equation}
Let $f(x_{0})<l<L$, then $\lim\limits_{\rho\rightarrow0}M_{l}(\rho)=L-l>0$ and $M_{l}(\rho)$
is bounded away from zero for sufficiently small $\rho$: there exist constants
$a, \rho_{\ast} \in (0,1)$ such that
\begin{equation}\label{eqMla}
M_{l}(\rho)>a>0 \ \ \text{for all} \ \rho\in(0,\rho_{\ast}).
\end{equation}
Let us prove the following inequality:
\begin{equation}\label{eq2.0}
\textrm{Cap}_{G}\big(\overline{B}_{\rho/32}(x_{0})\setminus\Omega, B_{\rho}(x_{0})  \big)
\leqslant a^{1-q} \Lambda(\gamma,3n,\rho)\rho^{n-1}
g\left( x_{0},\frac{M_{l}(\rho)-M_{l}(\rho/4)+2b_{0}\rho}{\rho} \right)
\end{equation}
for arbitrary fixed $\rho\in(0,\rho_{\ast})$.
To do this, we fix a function
$\eta\in C_{0}^{\infty}(B_{\rho/16}(x_{0}))$ such that
$0\leqslant\eta\leqslant1$, $\eta=1$ in $B_{3\rho/64}(x_{0})$ and
$|\nabla\eta|\leqslant 64/\rho$.
We set
$$
v_{l,\rho}=\frac{u_{l,\rho}\eta}{M_{l}(\rho)+2b_{0}\rho}
\quad \text{ and } \quad
E_{l,\rho/32}=\overline{B}_{\rho/32}(x_{0})\cap \overline{\{x\in \mathbb{R}^{n}:u_{l}(x)=0\}}.
$$
It is clear that $\overline{B}_{\rho/32}(x_{0})\setminus\Omega\subset E_{l,\rho/32}$ and
$v_{l,\rho}\in \mathfrak{M}(E_{l,\rho/32}, B_{\rho}(x_{0}))$,
therefore by Definition \ref{defcapacity} the following inequality holds:
\begin{equation}\label{eq2.1}
\textrm{Cap}_{G}\big(\overline{B}_{\rho/32}(x_{0})\setminus\Omega, B_{\rho}(x_{0})  \big)\leqslant
\textrm{Cap}_{G}\big(E_{l,\rho/32}, B_{\rho}(x_{0})  \big)
\leqslant
\int_{B_{\rho/16}(x_{0})} G(x, |\nabla v_{l,\rho}|)\,dx.
\end{equation}
%Âäîáŕâîę, ĺńëč $M_{l}(\rho)+2b_{0}\rho=1$, ňî íĺđŕâĺíńňâî \eqref{eq1.7}
The integral in \eqref{eq2.1} can be estimated as follows:
\begin{equation}\label{eq2.3}
\int_{\Omega_{l,\rho/16}}G(x, |\nabla v_{l,\rho}|)\,dx
\leqslant a^{1-q} \Lambda(\gamma, 3n,\rho)\rho^{n-1}
g\bigg( x_{0},\frac{M_{l}(\rho)-M_{l}(\rho/4)+2b_{0}\rho}{\rho} \bigg),
\end{equation}
that implies \eqref{eq2.0}.
%÷ňî âěĺńňĺ ń \eqref{eq2.1} âëĺ÷ĺň \eqref{eq2.0}.
Indeed, if $M_{l}(\rho)+2b_{0}\rho\geqslant1$, then by condition (${\rm g}_{2}$) and
\eqref{eq1.7} we have
$$
g\bigg(x, \frac{|\nabla (u_{l,\rho}\eta)|}{M_{l}(\rho)+2b_{0}\rho} \bigg)
\leqslant
(M_{l}(\rho)+2b_{0}\rho)^{1-p}
g(x, |\nabla (u_{l,\rho}\eta)|)
\ \ \text{in} \ \Omega_{l,\rho/16},
$$
$$
\begin{aligned}
&\int_{\Omega_{l,\rho/16}}G(x, |\nabla v_{l,\rho}|)\,dx
=
\int_{\Omega_{l,\rho/16}}
g\bigg(x, \frac{|\nabla (u_{l,\rho}\eta)|}{M_{l}(\rho)+2b_{0}\rho} \bigg)
\frac{|\nabla (u_{l,\rho}\eta)|}{M_{l}(\rho)+2b_{0}\rho}\,dx
\\
&\leqslant (M_{l}(\rho)+2b_{0}\rho)^{-p}
\int_{\Omega_{l,\rho/16}}
G(x, |\nabla(u_{l,\rho}\eta)|)\,dx
\\
&\leqslant
\Lambda(\gamma, 3n,\rho)(M_{l}(\rho)+2b_{0}\rho)^{1-p}\rho^{n-1}
g\bigg( x_{0},\frac{M_{l}(\rho)-M_{l}(\rho/4)+2b_{0}\rho}{\rho} \bigg)
\\
&\leqslant \Lambda(\gamma, 3n,\rho)\rho^{n-1}
g\bigg( x_{0},\frac{M_{l}(\rho)-M_{l}(\rho/4)+2b_{0}\rho}{\rho} \bigg),
\end{aligned}
$$
that implies \eqref{eq2.3}.
In the case $0< M_{l}(\rho)+2b_{0}\rho<1$, by conditions (${\rm g}_{1}$) and
(${\rm g}_{3}$), we have on the set $\{|\nabla (u_{l,\rho}\eta)|\geqslant b_{0}\}$:
$$
g\bigg(x, \frac{|\nabla (u_{l,\rho}\eta)|}{M_{l}(\rho)+2b_{0}\rho} \bigg)
\leqslant c_{1}(M_{l}(\rho)+2b_{0}\rho)^{1-q}
g(x, |\nabla (u_{l,\rho}\eta)|),
%\ \ \text{on} \ \{|\nabla (u_{l,\rho}\eta)|\geqslant b_{0}\},
$$
and on the set $\{|\nabla (u_{l,\rho}\eta)|< b_{0}\}$:
$$
\begin{aligned}
&g\bigg(x, \frac{|\nabla (u_{l,\rho}\eta)|}{M_{l}(\rho)+2b_{0}\rho} \bigg)
\leqslant
g\bigg(x, \frac{b_{0}}{M_{l}(\rho)+2b_{0}\rho} \bigg)
\\
&\leqslant
c_{1} (M_{l}(\rho)+2b_{0}\rho)^{1-q}g(x,b_{0})
\\
&\leqslant
c_{1} (M_{l}(\rho)+2b_{0}\rho)^{1-q}
g\bigg(x, \frac{M_{l}(\rho)-M_{l}(\rho/4)+2b_{0}\rho}{\rho}\bigg)
\\
&\leqslant
\gamma e^{\lambda(\rho)} (M_{l}(\rho)+2b_{0}\rho)^{1-q}
g\bigg(x_{0}, \frac{M_{l}(\rho)-M_{l}(\rho/4)+2b_{0}\rho}{\rho}\bigg).
%\ \ \text{on} \ \{|\nabla (u_{l,\rho}\eta)|< b_{0}\}.
\end{aligned}
$$
Using these relations, \eqref{eq1.7} and \eqref{eqMla}, we obtain
$$
\begin{aligned}
&\int_{\Omega_{l,\rho/16}}
G(x, |\nabla v_{l,\rho}|)\,dx
\\
&=
\int_{\Omega_{l,\rho/16}\cap\{|\nabla v_{l,\rho}|\geqslant b_{0}\}}G(x, |\nabla v_{l,\rho}|)\,dx+
\int_{\Omega_{l,\rho/16}\cap\{|\nabla v_{l,\rho}|< b_{0}\}}G(x, |\nabla v_{l,\rho}|)\,dx
\\
&\leqslant
c_{1}(M_{l}(\rho)+2b_{0}\rho)^{-q}
\int_{\Omega_{l,\rho/16}}
G(x, |\nabla (u_{l,\rho}\eta)|)\,dx
\\
&\phantom{\leqslant\,}+\gamma e^{\lambda(\rho)}\rho^{n-1}(M_{l}(\rho)+2b_{0}\rho)^{1-q}
g\bigg(x_{0}, \frac{M_{l}(\rho)-M_{l}(\rho/4)+2b_{0}\rho}{\rho}\bigg)
\\
&\leqslant
a^{1-q}\Lambda(\gamma, 3n,\rho)\rho^{n-1}
g\bigg( x_{0},\frac{M_{l}(\rho)-M_{l}(\rho/4)+2b_{0}\rho}{\rho} \bigg),
\end{aligned}
$$
which again leads to \eqref{eq2.3}.
Thus, inequality \eqref{eq2.3}, and hence \eqref{eq2.0}, are completely proved.

Now, we can rewrite inequality \eqref{eq2.0} in the following form:
$$
g^{-1}_{x_{0}}
\bigg(a^{q-1}\Lambda(-\gamma,3n,\rho)\,
\frac{\textrm{Cap}_{G}\big(\overline{B}_{\rho/32}(x_{0})
\setminus\Omega, B_{\rho}(x_{0}) \big)}{\rho^{n-1}}
 \bigg)
\leqslant  %
\frac{M_{l}(\rho)-M_{l}(\rho/4)+2b_{0}\rho}{\rho}.
$$
Hence, setting
$$
\begin{aligned}
{\rm w}&=g^{-1}_{x_{0}}
\bigg(\Lambda(-\gamma,3n,\rho)\,
\frac{\textrm{Cap}_{G}\big(\overline{B}_{\rho/32}(x_{0})
\setminus\Omega, B_{\rho}(x_{0}) \big)}{\rho^{n-1}}
 \bigg),
\\
{\rm v}&=g^{-1}_{x_{0}}
\bigg(a^{q-1}\Lambda(-\gamma,3n,\rho)\,
\frac{\textrm{Cap}_{G}\big(\overline{B}_{\rho/32}(x_{0})
\setminus\Omega, B_{\rho}(x_{0}) \big)}{\rho^{n-1}}
 \bigg),
\end{aligned}
$$
and using \eqref{gpineq}, we derive the following:
$$
g^{-1}_{x_{0}}
\bigg(\Lambda(-\gamma,3n,\rho)\,
\frac{\textrm{Cap}_{G}\big(\overline{B}_{\rho/32}(x_{0})
\setminus\Omega, B_{\rho}(x_{0}) \big)}{\rho^{n-1}}
 \bigg)\leqslant
 a^{-\frac{q-1}{p-1}}\,
\frac{M_{l}(\rho)-M_{l}(\rho/4)+2b_{0}\rho}{\rho}.
$$
It is readily verified that
$$
\int_{0}\frac{M_{l}(\rho)-M_{l}(\rho/4)+2b_{0}\rho}{\rho}\,d\rho
<+\infty,
$$
and therefore the previous inequality gives
$$
\int_{0}g^{-1}_{x_{0}}
\bigg(\Lambda(-\gamma,3n,\rho)\,
\frac{\textrm{Cap}_{G}\big(B_{\rho/32}(x_{0})\setminus\Omega, B_{\rho}(x_{0}) \big)}{\rho^{n-1}}
 \bigg)\,d\rho<+\infty,
$$
which contradicts condition \eqref{Winercond} of Theorem \ref{thregpoint}.
Consequently, hypothesis \eqref{Lfhyp} is not correct, and the inequality
$\lim\limits_{\rho\rightarrow0}\esssup\limits_{\Omega\cap B_{\rho}(x_{0})} u\leqslant f(x_{0})$
is correct.
The inequality
$f(x_{0})\leqslant \lim\limits_{\rho\rightarrow0}\essinf\limits_{\Omega\cap B_{\rho}(x_{0})} u$
is proved similarly, which, together with the previous one, implies the equality
$f(x_{0})=\lim\limits_{\Omega\ni x\rightarrow x_{0}}u$.
The proof is complete.

%%%%%%%%%%%%%%%%%%%%%%%%%%%%%%%%%%%%%%%%%%%%%%%%%%%%%%%%%%%%%%%%%%%%%%%%%%%%%%%%%%%%%%%%%%%%%%%%%%%%%%%%%%%%%%%%%%
%%%%%%%%%%%%%%%%%%%%%%%%%%%%%%%%%%%%%%%%%%%%%%%%%%%%%%%%%%%%%%%%%%%%%%%%%%%%%%%%%%%%%%%%%%%%%%%%%%%%%%%%%%%%%%%%%%%%%%%
%%%%%%%%%%%%%%%%%%%%%%%%%%%%%%%%%%%%%%%%%%%%%%%%%%%%%%%%%%%%%%%%%%%%%%%%%%%%%%%%%%%%%%%%%%%%%%%%%%%%%%%%%%%%%%%%%%%%%%%%%

\section{The weak Harnack inequality: proof of Theorem \ref{thweakHarnineq}}
\label{SectHarnproof}

In this section we prove Theorem \ref{thweakHarnineq}.
We assume that all the hypotheses and notation of this theorem
are in force. For definiteness, we also assume that $l\geqslant M_{f}(\rho)$ in \eqref{defulrho}.
In the case where $l\leqslant m_{f}(\rho)$, the proof is completely similar.
We need some inequalities and several lemmas.
First, we note simple analogues of Young's inequality:
\begin{equation}\label{gYoungineq1}
g(x,a)\,b\leqslant \varepsilon\, g(x,a)\,a+g(x,b/\varepsilon)\,b
\ \ \text{if} \ \varepsilon, a, b >0, \ x\in \overline{\Omega}.
\end{equation}
In fact, if $b\leqslant\varepsilon a$, then $g(x,a)\,b\leqslant \varepsilon\, g(x,a)\,a$,
and if $b>\varepsilon a$, then since the function ${\rm v}\rightarrow g(x, {\rm v})$
is increasing we have that $g(x,a)\,b\leqslant g(x,b/\varepsilon)\,b$,
which proves inequality \eqref{gYoungineq1}.

Next, we set
\begin{equation}\label{defG}
\mathcal{G}(x, {\rm w})=\int_{0}^{{\rm w}}g(x,{\rm v})\,d{\rm v}
\ \ \text{for} \ \ x\in\Omega, \ {\rm w}>0.
\end{equation}
The following inequalities hold:
\begin{equation}\label{G>gw}
\mathcal{G}(x,{\rm w})\geqslant \gamma\,G(x,{\rm w}) \ \
\text{for all} \ \ x\in\overline{\Omega}, \  {\rm w}\geqslant 2b_{0},
\end{equation}
\begin{equation}\label{gw>pG}
G(x,{\rm w})\geqslant p\,\mathcal{G}(x,{\rm w}) \ \
\text{for all} \ \ x\in \overline{\Omega}, \ {\rm w}> 0.
\end{equation}
Indeed, if $x\in\overline{\Omega}$ and ${\rm w}\geqslant 2b_{0}$, then
by \eqref{gqineq}, \eqref{deffuncgw} and \eqref{defG}, we have
$$
\mathcal{G}(x,{\rm w})= \int_{0}^{{\rm w}}g(x,{\rm v})\,d{\rm v}
\geqslant
\int_{b_{0}}^{{\rm w}}g(x,{\rm v})\,d{\rm v}\geqslant
\frac{g(x,{\rm w})}{c_{1}{\rm w}^{\,q-1}}
\int_{b_{0}}^{{\rm w}} {\rm v}^{\,q-1}d{\rm v}\geqslant
\frac{1-2^{-q}}{c_{1}q}\,G(x,{\rm w}),
$$
which implies \eqref{G>gw}.
Now, let $x\in \overline{\Omega}$ and ${\rm w}>0$ be arbitrary,
then by \eqref{gpineq}, \eqref{deffuncgw} and \eqref{defG} we obtain
$$
\mathcal{G}(x,{\rm w})= \int_{0}^{{\rm w}}g(x,{\rm v})\,d{\rm v}
\leqslant
\frac{g(x,{\rm w})}{{\rm w}^{p-1}}\int_{0}^{{\rm w}}
{\rm v}^{\,p-1}\,d{\rm v}=\frac{1}{p}\,g(x,{\rm w}){\rm w}=\frac{1}{p}\,G(x,{\rm w}),
$$
which yields \eqref{gw>pG}.

The rest of the lemmas in this section are successive stages in the proof of
Theorem \ref{thweakHarnineq}, which follows Trudinger's strategy \cite{TrudingerArch71}
adapted to Eq.\eqref{gellequation} near the boundary $\partial\Omega$
(see Remark \ref{remHarnineq}).
%Note that the traditional Moser method \cite{Moser1960, Moser1961}
%(as, for example, in \cite{BenHarHastKarp, BenKhlifi}) is inapplicable here,
%since checking whether
%the logarithm of solutions belongs to the BMO space is a problem for the non-logarithmic condition
%\eqref{eqnonlogcond} with unbounded $\lambda(r)$.

\begin{lemma}\label{lem2.1}
%Let all the assumptions of Theorem \ref{thweakHarnineq} be fulfilled.
There exists positive constant $\gamma$ depending only on the known data such that
\begin{equation}\label{ShSkr2.4}
\exp\bigg( \fint_{B_{\rho/2}(x_{0})}\ln u_{l,\rho}\,dx \bigg)
\leqslant \Lambda(\gamma,3n,\rho)
\Big(M_{l}(\rho)-M_{l}(\rho/4)+2b_{0}\rho\Big).
\end{equation}
\end{lemma}
\begin{proof}
Let
\begin{equation}\label{deffnctionw}
w=\ln\frac{\varkappa}{ u_{l,\rho} },
\end{equation}
where the function $u_{l,\rho}$ is defined by \eqref{defulrho}
and the constant $\varkappa$ is defined by the condition
%$(w)_{x_{0},\rho}=\fint_{B_{\rho}(x_{0})}w\,dx=0$,
%i.e.
\begin{equation}\label{overline{u}}
(w)_{x_{0},\rho/2}=\fint_{B_{\rho/2}(x_{0})}w\,dx=0,
\ \ \text{i.e.} \ \
\varkappa=\exp \bigg( \fint_{B_{\rho/2}(x_{0})}
\ln u_{l,\rho}\,dx \bigg).
\end{equation}
Taking into account \eqref{defulrho}, \eqref{deffnctionw} and \eqref{overline{u}},
it is easy to see that inequality \eqref{ShSkr2.4} is equivalent to the following estimate:
\begin{equation}\label{estmaxw}
\esssup\limits_{B_{\rho/4}(x_{0})} w\leqslant\gamma e^{3n\lambda(\rho)}.
\end{equation}
The idea of using an upper bound of auxiliary logarithmic functions goes back
to Moser \cite{Moser1960, Moser1961} and has become a useful tool in the qualitative
theory of partial differential equations \cite{GilbTr, LadUr}.
In this paper, in order to prove \eqref{estmaxw}, we use  the approach of  De\,Giorgi
\cite{DeGiorgi1957, LadUr} in the spirit of our recent studies \cite{ShSkrVoit, SkrVoitarXiv20}.

We fix $\sigma\in(0,1)$, and for any $\rho/4\leqslant r<r(1+\sigma)\leqslant\rho/2$ we
take a function  $\zeta\in C_{0}^{\infty}(B_{r(1+\sigma)}(x_{0}))$,
$0\leqslant\zeta\leqslant1$, $\zeta=1$ in $B_{r}(x_{0})$ and
$|\nabla \zeta|\leqslant(\sigma r)^{-1}$.
Let
\begin{equation}\label{kcondition}
k\geqslant \gamma e^{2n\lambda(\rho)}
\bigg( \fint_{B_{\rho/2}(x_{0})} |w|^{\frac{n}{n-1}}\,dx\bigg)^{\frac{n-1}{n}}+1.
\end{equation}
From \eqref{deffnctionw}, \eqref{kcondition} it follows that
\begin{equation}\label{kmoreln}
k>\ln \dfrac{\varkappa}{M_{l}(\rho)+2b_{0}\rho}
=\esssup\limits_{\partial\Omega\cap B_{\rho}(x_{0})}w.
\end{equation}
We test \eqref{gelintidentity} by the function
\begin{equation}\label{testfunct1}
\varphi=\begin{cases}
\dfrac{u_{l,\rho}\,(w-k)_{+}}
{\mathcal{G}\left(x_{0}, u_{l,\rho}/\rho \right)}\,\zeta^{\,q}
& \text{on } \Omega\cap B_{r(1+\sigma)}(x_{0}), \\
0
& \text{otherwise}.
\end{cases}
\end{equation}
Since we are dealing with bounded solutions,
then this function % $\varphi$ in \eqref{testfunct1}
and all other test functions used in the article
belong to $W^{1,G}_{0}(\Omega)$.
This is a consequence of conditions {\rm (${\rm g}_{0}$)}, {\rm (${\rm g}_{1}$)},
the result of Marcus and Mizel \cite[Theorem~2]{MarcMizel} and the notion of the weak inequality
on the boundary $\partial\Omega$ \cite{GarZiem, LitStWei}.
So, after substitution \eqref{testfunct1} into \eqref{gelintidentity}, we have
$$
\begin{aligned}
&\int_{A_{k,r(1+\sigma)}}\frac{G(x,|\nabla u|)}
{\mathcal{G}\left( x_{0}, u_{l,\rho}/\rho \right)}\,\zeta^{\,q}\,dx
\\
&+\int_{A_{k,r(1+\sigma)}}\frac{G(x,|\nabla u|)}
{\mathcal{G}\left( x_{0}, u_{l,\rho}/\rho \right)}
\left\{ \frac{G(x_{0}, u_{l,\rho}/\rho)}
{\mathcal{G}(x_{0}, u_{l,\rho}/\rho)}-1 \right\}(w-k)_{+}\,\zeta^{\,q}\,dx
\\
&\leqslant \frac{\gamma}{\sigma}\int_{A_{k,r(1+\sigma)}}
\frac{g(x,|\nabla u|)}
{\mathcal{G}\left( x_{0}, u_{l,\rho}/\rho \right)}\,
\frac{u_{l,\rho}}{\rho}\,(w-k)_{+}\,\zeta^{\,q-1}\,dx,
\end{aligned}
$$
where $A_{k,r(1+\sigma)}=\Omega\cap B_{r(1+\sigma)}(x_{0})\cap \{w>k\}$
and  the embedding $A_{k,r(1+\sigma)}\subset \Omega_{l,r(1+\sigma)}$
is true due to \eqref{kmoreln} (see \eqref{defOmegalr} for the definition of
$\Omega_{l,r(1+\sigma)}$).
By \eqref{gw>pG}, the value in curly brackets is estimated from below as follows:
\begin{equation}\label{GmthclGp-1}
\frac{G(x_{0}, u_{l,\rho}/\rho)}
{\mathcal{G}(x_{0}, u_{l,\rho}/\rho)}-1\geqslant p-1,
\end{equation}
and therefore
\begin{equation}\label{Voit1}
\begin{aligned}
\int_{A_{k,r(1+\sigma)}}\frac{G(x,|\nabla u|)}
{\mathcal{G}\left( x_{0}, u_{l,\rho}/\rho \right)}\,\zeta^{\,q}\,dx+
(p-1) \int_{A_{k,r(1+\sigma)}}\frac{G(x,|\nabla u|)}
{\mathcal{G}\left( x_{0}, u_{l,\rho}/\rho \right)}\,(w-k)_{+}\,\zeta^{\,q}\,dx
\\
\leqslant \gamma\int_{A_{k,r(1+\sigma)}}
\frac{g(x,|\nabla u|)}
{\mathcal{G}\left( x_{0}, u_{l,\rho}/\rho \right)}\,
\frac{u_{l,\rho}}{\sigma\rho\,\zeta}\,(w-k)_{+}\,\zeta^{\,q}\,dx.
\end{aligned}
\end{equation}
We use inequality \eqref{gYoungineq1} with $a=|\nabla u|$,
$b=\dfrac{u_{l,\rho}}{\sigma\rho\,\zeta}$ and
sufficiently small $\varepsilon>0$, and then \eqref{G>gw} with
${\rm w}=u_{l,\rho}/\rho$,
to estimate from above the right-hand side of \eqref{Voit1}:
\begin{multline*}\label{Voit1}
\gamma\int_{A_{k,r(1+\sigma)}}
\frac{g(x,|\nabla u|)}
{\mathcal{G}\left( x_{0}, u_{l,\rho}/\rho \right)}\,
\frac{u_{l,\rho}}{\sigma\rho\,\zeta}\,(w-k)_{+}\,\zeta^{\,q}\,dx
\\
\leqslant
\frac{p-1}{2}\int_{A_{k,r(1+\sigma)}}
\frac{G(x,|\nabla u|)}
{\mathcal{G}\left( x_{0}, u_{l,\rho}/\rho \right)}\,
(w-k)_{+}\,\zeta^{\,q}\,dx
\\
+\frac{\gamma}{\sigma}\int_{A_{k,r(1+\sigma)}}
\frac{g\big( x,  \frac{\gamma\,u_{l,\rho}}{\sigma\rho\,\zeta} \big)}
{g\left( x_{0}, u_{l,\rho}/\rho \right)}(w-k)_{+}\,\zeta^{\,q-1}\,dx.
\end{multline*}
Combining this inequality and \eqref{Voit1}, %and using \eqref{gw>pG} with ${\rm w}=|\nabla u|$,
we obtain that
\begin{equation}\label{ineqGGgg}
\int_{A_{k,r(1+\sigma)}}
\frac{G(x,|\nabla u|)}{\mathcal{G}\left( x_{0}, u_{l,\rho}/\rho \right)}\,\zeta^{\,q}\,dx
\leqslant
\frac{\gamma}{\sigma}\int_{A_{k,r(1+\sigma)}}
\frac{g\big( x,  \frac{\gamma\,u_{l,\rho}}{\sigma\rho\,\zeta} \big)}
{g\left( x_{0}, u_{l,\rho}/\rho \right)}\,(w-k)_{+}\,\zeta^{\,q-1}\,dx.
\end{equation}
Since
$
\dfrac{\gamma\,u_{l,\rho}}{\sigma\rho\,\zeta}
\geqslant
\dfrac{u_{l,\rho}}{\rho}
\geqslant 2b_{0}
$
and $|x-x_{0}|<r(1+\sigma)< \rho$\, for $x\in A_{k,r(1+\sigma)}$,
then using conditions (${\rm g}_{1}$) and (${\rm g}_{3}$), we get
that for all $x\in  A_{k,r(1+\sigma)}$ there holds:
$$
g\left(x,  \frac{\gamma\,u_{l,\rho}}{\sigma\rho\,\zeta}\right)
\leqslant\gamma\,(\sigma\zeta)^{1-q}\,
g\left(x,u_{l,\rho}/\rho\right)
\leqslant
\gamma\,(\sigma\zeta)^{1-q}\,e^{\lambda(\rho)}
g\left(x_{0},u_{l,\rho}/\rho \right).
$$
So, from \eqref{ineqGGgg} we obtain
\begin{equation}\label{ShSkr2.5}
\int_{A_{k,r(1+\sigma)}}
\frac{G(x,|\nabla u|)}{\mathcal{G}\left( x_{0}, u_{l,\rho}/\rho \right)}\,\zeta^{\,q}\,dx
\leqslant
\gamma\,\sigma^{-q}\,e^{\lambda(\rho)}
\int_{A_{k,r(1+\sigma)}} (w-k)_{+}\,dx.
\end{equation}

To estimate the term on the left-hand side of \eqref{ShSkr2.5}, we use \eqref{gYoungineq1} with
$\varepsilon=1$, $a=u_{l,\rho}/\rho$, $b=|\nabla u|$, assumption
(${\rm g}_{3}$), the definitions of the functions $G$, $\mathcal{G}$, $w$
(see equalities \eqref{deffuncgw}, \eqref{defG} and \eqref{deffnctionw}, respectively)
and \eqref{gw>pG}:
\begin{equation}\label{eqShSkr2.6}
\begin{aligned}
\int_{A_{k,r(1+\sigma)}}|\nabla w|\,\zeta^{\,q}\,dx
&=
\int_{A_{k,r(1+\sigma)}} \frac{|\nabla u|}{u_{l,\rho}}\,
\frac{g\left(x,u_{l,\rho}/\rho\right)}{g\left(x,u_{l,\rho}/\rho\right)}\,\zeta^{\,q}\,dx
\\
&\leqslant \frac{1}{\rho}\,|A_{k,r(1+\sigma)}|+
\frac{1}{\rho}\int\limits_{A_{k,r(1+\sigma)}}
\frac{G(x,|\nabla u|)}{G\left( x, u_{l,\rho}/\rho \right)}\,\zeta^{\,q}\,dx
\\
&\leqslant \frac{1}{\rho}\,|A_{k,r(1+\sigma)}|+
\gamma \frac{e^{\lambda(\rho)}}{\rho}
\int_{A_{k,r(1+\sigma)}}
\frac{G(x,|\nabla u|)}{\mathcal{G}\left( x_{0}, u_{l,\rho}/\rho \right)}\,\zeta^{\,q}\,dx.
\end{aligned}
\end{equation}
Collecting \eqref{ShSkr2.5} and \eqref{eqShSkr2.6}, we obtain
$$
\int_{A_{k,r(1+\sigma)}}|\nabla w|\,\zeta^{\,q}\,dx\leqslant
\frac{\gamma}{\sigma^{q}}\,\frac{e^{2\lambda(\rho)}}{\rho}
\bigg(|A_{k,r(1+\sigma)}|+ \int_{A_{k,r(1+\sigma)}} (w-k)_{+}\,dx\bigg).
$$
From this, using Sobolev's embedding theorem, standard iteration arguments
(see e.g. \cite[Section~2, Theorem~5.3]{LadUr})
and condition \eqref{kcondition} on $k$, we obtain that
\begin{equation}\label{ShSkr2.7}
\esssup\limits_{B_{\rho/4}(x_{0})}w\leqslant
\gamma\,e^{2n\lambda(\rho)}
\bigg(\fint_{B_{\rho/2}(x_{0})}
|w|^{\,\frac{n}{n-1}}\,dx\bigg)^{\frac{n-1}{n}}+1.
\end{equation}

In order to estimate the right-hand side of \eqref{ShSkr2.7} we use the Poincar\'{e} inequality.
By our choice of $\varkappa$ in \eqref{overline{u}} we have
\begin{equation}\label{Poincarewn/n-1}
\begin{aligned}
\bigg(\fint_{B_{\rho/2}(x_{0})}
|w|^{\,\frac{n}{n-1}}\,dx\bigg)^{\frac{n-1}{n}}
=\bigg(&\fint_{B_{\rho/2}(x_{0})}
|w-(w)_{x_{0},\rho/2}|^{\,\frac{n}{n-1}}\,dx\bigg)^{\frac{n-1}{n}}
\\
&\leqslant \gamma\,\rho^{1-n}
\int_{B_{\rho/2}(x_{0})}|\nabla w|\,dx.
\end{aligned}
\end{equation}
Next, similarly to \eqref{eqShSkr2.6}, we have
\begin{equation}\label{intDwbyintGG}
%\begin{aligned}
\int_{B_{\rho/2}(x_{0})}|\nabla w|\,dx
\leqslant
\int_{B_{\rho}(x_{0})}|\nabla w|\,\zeta^{\,q}\,dx
\leqslant \gamma\rho^{n-1}+\gamma\,\frac{e^{\lambda(\rho)}}{\rho}
\int_{\Omega_{l,\rho}}
\frac{G(x,|\nabla u|)}
{\mathcal{G}\left( x_{0}, u_{l,\rho}/\rho \right)}\,\zeta^{\,q}\,dx,
%\end{aligned}
\end{equation}
where we have
%$\Omega_{l,\rho}=\{u>l\}\cap B_{\rho}(x_{0})$,
$\zeta\in C_{0}^{\infty}(B_{\rho}(x_{0}))$,
$0\leqslant\zeta\leqslant1$, $\zeta=1$ in $B_{\rho/2}(x_{0})$, and
$|\nabla\zeta|\leqslant 2/\rho$.
In addition, testing \eqref{gelintidentity} by
$$
\varphi=\begin{cases}
\bigg(\dfrac{u_{l,\rho}}{\mathcal{G}\left( x_{0}, u_{l,\rho}/\rho\right)}
-\dfrac{M_{l}(\rho)+2b_{0}\rho}
{\mathcal{G}\left( x_{0}, M_{l}(\rho)\rho^{-1}+2b_{0}\right)}\bigg)
\zeta^{\,q} & \text{on } \Omega\cap B_{\rho}(x_{0}), \\
0
& \text{otherwise},
\end{cases}
$$
similarly to \eqref{ShSkr2.5}, we obtain
\begin{equation}\label{ShSkrVoit2.8}
\int_{\Omega_{l,\rho}}
\frac{G(x,|\nabla u|)}{\mathcal{G}\left( x_{0}, u_{l,\rho}/\rho \right)}\,\zeta^{\,q}\,dx
\leqslant \gamma \rho^{n}e^{\lambda(\rho)}.
\end{equation}
Now, collecting \eqref{ShSkr2.7}--\eqref{ShSkrVoit2.8},
we arrive at the required inequality \eqref{estmaxw}.
The proof of the lemma is complete.
\end{proof}

\begin{lemma}\label{ShSkrlem2.2}
%Under the assumptions of Theorem \ref{thweakHarnineq}
There exists a positive number
$\delta_{0}=\delta_{0}(\rho)$ depending only on the data and $\rho$,
such that
\begin{equation}\label{eqShSkr2.9}
\bigg(\fint_{B_{\rho/4}(x_{0})}
u_{l,\rho}^{\delta_{0}}\,dx\bigg)^{1/\delta_{0}}
\leqslant \Lambda(\gamma,2n,\rho)\exp
\bigg( \fint_{B_{\rho/2}(x_{0})}
\ln u_{l,\rho}\,dx\bigg).
\end{equation}
\end{lemma}
\begin{proof}
Let's fix $\sigma\in(0,1)$ and for any $\rho/4\leqslant r<r(1+\sigma)\leqslant \rho/2$
consider the function $\zeta\in C_{0}^{\infty}\big(B_{r(1+\sigma)}(x_{0})\big)$,
$0\leqslant\zeta\leqslant 1$, $\zeta=1$ in $B_{r}(x_{0})$,
$|\nabla \zeta|\leqslant (\sigma r)^{-1}$.
We define
\begin{equation}\label{defvmu}
v=\ln\frac{u_{l,\rho}}{\varkappa}, \quad
v_{\mu}=\max\{v,\,\mu\}, \quad \mu>0.
\end{equation}
Testing \eqref{gelintidentity} by
$$
\varphi=\begin{cases}
v_{\mu}^{s-1}\,\bigg(\dfrac{u_{l,\rho}}{\mathcal{G}\left( x_{0}, u_{l,\rho}/\rho\right)}
-\dfrac{M_{l}(\rho)+2b_{0}\rho}
{\mathcal{G}\left( x_{0}, M_{l}(\rho)\rho^{-1}+2b_{0}\right)}\bigg)\,
\zeta^{\,\theta} & \text{on } \Omega\cap B_{r(1+\sigma)}(x_{0}), \\
0
& \text{otherwise},
\end{cases}
$$
where $s\geqslant 1$, $\theta\geqslant q$,
and using \eqref{GmthclGp-1}, we have
$$
\begin{aligned}
( p-1)&\int_{\Omega_{l,r(1+\sigma)}}
\frac{G(x,|\nabla u|)}{\mathcal{G}\left( x_{0}, u_{l,\rho}/\rho \right)}\,
v_{\mu}^{s-1}\,\zeta^{\,\theta}\,dx
\\
&\leqslant
(s-1)\int_{\{v>\mu\}\cap\Omega_{l,r(1+\sigma)}}
\frac{G(x,|\nabla u|)}{\mathcal{G}\left( x_{0}, u_{l,\rho}/\rho \right)}\,
v_{\mu}^{s-2}\,\zeta^{\,\theta}\,dx
\\
& \ \ \ \ \ \ +\gamma\, \theta\int_{\Omega_{l,r(1+\sigma)}}
\frac{g(x,|\nabla u|)}
{\mathcal{G}\left( x_{0}, u_{l,\rho}/\rho \right)}\,
\frac{u_{l,\rho}}{\sigma\rho\,\zeta}\,
 v_{\mu}^{s-1}\,\zeta^{\,\theta}\,dx.
\end{aligned}
$$
%where $\Omega_{l,r(1+\sigma)}=\{u>l\}\cap B_{r(1+\sigma)}(x_{0}) $.
Choosing $\mu$ from the condition $\dfrac{s}{\mu}=\dfrac{p-1}{2}$
and using inequalities \eqref{gYoungineq1}, \eqref{G>gw} and conditions (${\rm g}_{1}$)
and (${\rm g}_{3}$) similarly to the derivation of \eqref{ShSkr2.5},
from the previous we obtain
\begin{equation}\label{ShSkr2.10}
\int_{\Omega_{l,r(1+\sigma)}}
\frac{G(x,|\nabla u|)}{\mathcal{G}\left( x_{0}, u_{l,\rho}/\rho \right)}\,
v_{\mu}^{s-1}\,\zeta^{\,\theta}\,dx
\leqslant \frac{\gamma\,\theta^{\gamma}e^{\lambda(\rho)}}{\sigma^{q}}
\int_{B_{r(1+\sigma)}(x_{0})} v_{\mu}^{s-1}\,\zeta^{\,\theta-q}\,dx.
\end{equation}
%where $\Omega_{r(1+\sigma)}=\Omega\cap B_{r(1+\sigma)}(x_{0})$.
Estimating from below the term on the left-hand side of \eqref{ShSkr2.10}, similarly to \eqref{eqShSkr2.6},
we obtain
\begin{multline*}
\int_{B_{r(1+\sigma)}(x_{0})}
|\nabla v_{\mu}|\,v_{\mu}^{s-1}\,\zeta^{\,\theta}\,dx
\leqslant
\int_{\Omega_{l,r(1+\sigma)}}
\frac{|\nabla u|}{u_{l,\rho}}\,v_{\mu}^{s-1}\,\zeta^{\,\theta}\,dx
\\
\leqslant
\frac{\gamma\,\theta^{\gamma}}{\sigma^{q}}\,\frac{e^{2\lambda(\rho)}}{\rho}
\int_{B_{r(1+\sigma)}(x_{0})} v_{\mu}^{s-1}\,\zeta^{\,\theta-q}\,dx
\leqslant
\frac{\gamma\,\theta^{\gamma}}{\sigma^{q}}\,\frac{e^{2\lambda(\rho)}}{\rho}
\int_{B_{r(1+\sigma)}(x_{0})} v_{\mu}^{s}\,\zeta^{\,\theta-q}\,dx.
\end{multline*}
Using Sobolev's embedding theorem from this we have
\begin{equation}\label{ShSkr2.11}
\fint_{B_{r}(x_{0})}
v_{\mu}^{\frac{sn}{n-1}}dx\leqslant
\bigg( \frac{\gamma s\,e^{2\lambda(\rho)}}{\sigma^{q}}\,
\fint_{B_{r(1+\sigma)}(x_{0})}v_{\mu}^{s}\, dx\bigg)^{\frac{n}{n-1}}.
\end{equation}
For $j=0,1,2,\ldots$ , we define the following sequences:
%$r_{j}=\dfrac{\rho}{2}(1+2^{-j})$, $B_{j}=B_{r_{j}}(x_{0})$,
\begin{equation}\label{defmuj}
\begin{aligned}
r_{j}&=\dfrac{\rho}{4}(1+2^{-j}), \quad \ \ B_{j}=B_{r_{j}}(x_{0}),
\\
s_{j}&=\bigg(\frac{n}{n-1}\bigg)^{j+1},\quad
\mu_{j}=\frac{2s_{j}}{p-1}, \quad
y_{j}=\bigg(\fint_{B_{j}}
v_{\mu_{j}}^{s_{j}}\,dx\bigg)^{1/s_{j}}.
\end{aligned}
\end{equation}
Then inequality \eqref{ShSkr2.11} can be rewritten in the form
\begin{equation}\label{ShSkr2.12}
y_{j+1}\leqslant \left(\gamma\,2^{jq}
s_{j}\,e^{2\lambda(\rho)}\right)^{1/s_{j}} y_{j},
\quad j=0,1,2,\ldots.
\end{equation}
In addition, by \eqref{deffnctionw}, \eqref{defvmu}, \eqref{Poincarewn/n-1}--\eqref{ShSkrVoit2.8},
for $j=0$, we have
\begin{equation}\label{ShSkr2.13}
y_{0}\leqslant \gamma\mu_{0}+\gamma
\bigg( \fint_{B_{\rho/2}(x_{0})}|w|^{\frac{n}{n-1}}\,dx \bigg)^{\frac{n-1}{n}}
\leqslant \gamma e^{2\lambda(\rho)}.
\end{equation}
Iterating \eqref{ShSkr2.12} and taking into account \eqref{ShSkr2.13},
for $j=0,1,2,\ldots$, we have
\begin{equation}\label{ShSkr2.14}
y_{j+1}
\leqslant y_{0}
\gamma^{\,\sum\limits_{i=0}^{j} \frac{1}{s_{i}}}\,
2^{q\sum\limits_{i=1}^{j}\frac{i}{s_{i}}}
\Big(\frac{n}{n-1}\Big)^{\sum\limits_{i=0}^{j}\frac{i+1}{s_{i}}}
\exp\bigg(2\lambda(\rho) \sum\limits_{i=0}^{j}\frac{1}{s_{i}} \bigg)
\leqslant \gamma e^{2n\lambda(\rho)}.
\end{equation}
Let $m\in \mathbb{N}$ be arbitrary, then there exists $j\geqslant 1$ such that
$s_{j-1}<m\leqslant s_{j}$.
Using H\"{o}lder's inequality, %and Stirling's ($m!>\sqrt{2\pi m}\,(m/e)^{m}$)
from \eqref{defmuj}, \eqref{ShSkr2.14} we obtain
$$
\fint_{B_{\rho/4}(x_{0})}
\frac{v_{+}^{\,m}}{m!}\,dx
\leqslant
\fint_{B_{\rho/4}(x_{0})}
\frac{v_{\mu_{j}}^{\,m}}{m!}\,dx
\leqslant
\frac{\gamma\,y_{j}^{m}}{m!}
\leqslant \frac{\gamma^{m+1}}{m!}\,e^{2nm\lambda(\rho)}
\leqslant \gamma^{m+1}e^{2nm\lambda(\rho)}.
$$
Choosing $\delta_{0}=\delta_{0}(\rho)$ from the condition
\begin{equation}\label{ShSkr2.15}
\delta_{0}=\frac{1}{2\gamma}\,e^{-2n\lambda(\rho)},
\end{equation}
from the previous we have
$$
\fint_{B_{\rho/4}(x_{0})}
\frac{(\delta_{0}v_{+})^{m}}{m\,!}\,dx\leqslant \gamma\,2^{-m},
$$
which implies that
$$
\fint_{B_{\rho/4}(x_{0})}
e^{\delta_{0}v}\,dx
\leqslant
\fint_{B_{\rho/4}(x_{0})}
e^{\delta_{0}v_{+}}\,dx
\leqslant
\sum\limits_{m=0}^{\infty}\,
\fint_{B_{\rho/4}(x_{0})}
\frac{(\delta_{0}v_{+})^{m}}{m!}\,dx\leqslant 2\gamma.
$$
From this, since
$e^{\delta_{0}v}= (u_{l,\rho}/\varkappa)^{\delta_{0}}$
we have
$$
\bigg(\fint_{B_{\rho/4}(x_{0})}
u_{l,\rho}^{\,\delta_{0}}\,dx\bigg)^{1/\delta_{0}}
\leqslant (2\gamma)^{1/\delta_{0}} \varkappa
\leqslant \Lambda(\gamma, 2n, \rho)\,\varkappa,
$$
that together with \eqref{overline{u}} yields the desired inequality \eqref{eqShSkr2.9}.
This completes the proof of the lemma.
\end{proof}

The next lemma is a simple consequence of Lemmas \ref{lem2.1}
and \ref{ShSkrlem2.2}.

\begin{lemma}\label{lemShSkrVoit2.3}
The following inequality holds:
\begin{equation}\label{ShSkr2.16}
\bigg(\fint_{B_{\rho/4}(x_{0})}
g^{\delta_{1}}\left(x_{0}, u_{l,\rho}/\rho\right)\,dx\bigg)
^{1/\delta_{1}}
\leqslant \Lambda(\gamma, 3n, \rho)\,
g\left(x_{0}, \frac{M_{l}(\rho)-M_{l}(\rho/4)+2b_{0}\rho}{\rho}\right),
\end{equation}
where
\begin{equation}\label{defdelta1}
\delta_{1}=\delta_{0}/(q-1),
\end{equation}
and $\delta_{0}$ is defined by \eqref{ShSkr2.15}.
\end{lemma}
\begin{proof}
By condition (${\rm g}_{1}$) we have
\begin{multline*}
\fint_{B_{\rho/4}(x_{0})}
\frac{g^{\delta_{1}}\left(x_{0}, u_{l,\rho}/\rho\right)}
{g^{\delta_{1}}\left(x_{0}, \frac{M_{l}(\rho)-M_{l}(\rho/4)+2b_{0}\rho}{\rho}\right)}\,dx
\leqslant c_{1}^{\delta_{1}}
\fint_{B_{\rho/4}(x_{0})}
\left(\frac{u_{l,\rho}}{M_{l}(\rho)-M_{l}(\rho/4)+2b_{0}\rho}\right)
^{\delta_{0}}\,dx.
\end{multline*}
By Lemmas \ref{lem2.1} and \ref{ShSkrlem2.2} the right-hand side of
this inequality is estimated from above as follows:
$$
\fint_{B_{\rho/4}(x_{0})}
\left(\frac{u_{l,\rho}}{M_{l}(\rho)-M_{l}(\rho/4)+2b_{0}\rho}\right)
^{\delta_{0}}\,dx
\leqslant \Lambda(\gamma,3n,\rho),
$$
which proves the lemma.
\end{proof}

To complete the proof of Theorem \ref{thweakHarnineq} we need the following lemma.

\begin{lemma}[inverse H\"{o}lder inequality]\label{sk.lem 2.4}
%Let the assumptions  of Theorem \ref{thweakHarnineq} be fulfilled,
Let $\delta_{1}\leqslant s<n/(n-1)$,
where the number $\delta_{1}$ is defined by \eqref{ShSkr2.15} and \eqref{defdelta1}.
Then the following inequality holds:
\begin{equation}\label{sk 2.17}
\bigg(\fint_{B_{\rho/8}(x_0)} g^s\left(x_0, u_{l,\rho}/\rho\right) dx
\bigg)^{1/s}
\leqslant
\Lambda(\gamma, 2n+1, \rho)
\bigg(\fint_{B_{\rho/4}(x_0)} g^{\delta_1}
\left(x_0, u_{l,\rho}/\rho\right) dx \bigg)^{1/\delta_1}.
\end{equation}
\end{lemma}
\begin{proof}
We set
\begin{equation}\label{defpsi}
\psi(x, {\rm w})={\rm w}^{-1}\mathcal{G}(x, {\rm w}) \ \ \text{for} \
x\in\overline{\Omega}, \ {\rm w}>0,
\end{equation}
 and note that by \eqref{G>gw} and  \eqref{gw>pG} we have
\begin{equation}\label{g<gammapsi}
g(x,{\rm w})\leqslant \gamma\,\psi(x,{\rm w}) \ \
\text{for all} \ \ x\in\overline{\Omega}, \  {\rm w}\geqslant 2b_{0},
\end{equation}
\begin{equation}\label{psi<g}
\psi(x, {\rm w})\leqslant \frac{1}{p}\, g(x, {\rm w}) \ \
\text{for all} \ \ x\in \overline{\Omega}, \ {\rm w}> 0,
\end{equation}
which gives
\begin{equation}\label{psi'<psi/w}
\psi'_{{\rm w}} (x, {\rm w})\leqslant \gamma\,
\frac{\psi(x, {\rm w})}{{\rm w}} \ \
\text{for all} \ \ x\in\overline{\Omega}, \  {\rm w}\geqslant 2b_{0},
\end{equation}
\begin{equation}\label{psi'>psi/w}
\psi'_{{\rm w}} (x, {\rm w})
=\frac{g(x, {\rm w})-\psi(x, {\rm w})}{{\rm w}}
\geqslant (p-1)\, \frac{\psi(x, {\rm w})}{{\rm w}}
\ \
\text{for all} \ \ x\in \overline{\Omega}, \ {\rm w}> 0.
\end{equation}
We need a Cacciopoli-type inequality for negative powers of
$\psi\left(x_0, u_{l,\rho}/\rho\right)$.
To establish it, we fix $\sigma\in (0, 1)$ and $r>0$ such that
$\rho/8 \leqslant r< r(1+\sigma)\leqslant \rho/4$, and
take a function $\zeta\in C_0^{\infty}\left(B_{r(1+\sigma)}(x_0)\right)$,
$0\leqslant \zeta \leqslant 1$, $\zeta=1$ in $B_{r}(x_0)$,
$|\nabla \zeta|\leqslant (\sigma r)^{-1}$.
Testing \eqref{gelintidentity} by
$$
\varphi=\begin{cases}
\left[\psi^{-\tau} (x_{0}, u_{l,\rho}/\rho)
-\psi^{-\tau}(x_{0}, \rho^{-1}M_{l}(\rho)+2b_{0})
\right]
\zeta^{\,\theta} & \text{on } \Omega\cap B_{r(1+\sigma)}(x_{0}), \\
0
& \text{otherwise},
\end{cases}
$$
where $0<\tau<1$, $\theta\geqslant q$,
and using \eqref{psi'>psi/w} and the properties of $\zeta$, we obtain
the following inequality:
\begin{equation*}
\begin{aligned}
(p-1)\,\tau \int_{\Omega_{l,r(1+\sigma)}}
&\psi^{-\tau}\left(x_0, u_{l,\rho}/\rho\right)
\frac{G(x, |\nabla u|)}{u_{l,\rho}}\,\zeta^{\,\theta}\, dx
\\
\leqslant\frac{\gamma\, \theta }{\sigma\rho}\int_{\Omega_{l,r(1+\sigma)}}
&\psi^{-\tau} \left(x_0, u_{l,\rho}/\rho\right) g(x, |\nabla u|)\,\zeta^{\,\theta-1}\, dx,
\end{aligned}
\end{equation*}
%where $\Omega_{l,r(1+\sigma)}=\{u>l\}\cap B_{r(1+\sigma)}(x_0)$,
which by \eqref{gYoungineq1}, (${\rm g}_{1}$), (${\rm g}_{3}$) and \eqref{g<gammapsi} implies
the following:
\begin{equation}\label{sk 2.18}
\begin{aligned}
&\int_{\Omega_{l,r(1+\sigma)}}
\psi^{-\tau}\left(x_0,u_{l,\rho}/\rho\right)
\frac{G(x,|\nabla u|)}{u_{l,\rho}}\,\zeta^{\,\theta}\,dx
 \\
& \leqslant
\frac{\gamma\, \theta^{\,q}}{(\sigma \tau)^{q}}\,
\frac{e^{\lambda(\rho)}}{\rho}
\int_{B_{r(1+\sigma)}(x_0)} \psi^{1-\tau}
\left(x_0, u_{l,\rho}/\rho\right) \zeta^{\,\theta-q}\, dx.
\end{aligned}
\end{equation}

Basing on inequality \eqref{sk 2.18}, we organize Moser-type iterations
for the function $\psi\left(x_0, u_{l,\rho}/\rho\right)$.
To do this, we fix $0<t<n/(n-1)$ and $\vartheta\geqslant nq/(n-1)$,
then by Sobolev's inequality and by \eqref{psi'<psi/w} and \eqref{psi<g}, we obtain
\begin{equation}\label{sk 2.19}
\begin{aligned}
\bigg(&\int_{B_{r(1+\sigma)}(x_0)}
\psi^{\,t}\left(x_0, u_{l,\rho}/\rho\right)\zeta^{\,\vartheta}\, dx\bigg)^{\frac{n-1}{n}}
\\
&\leqslant \gamma  \int_{B_{r(1+\sigma)}(x_0)}
\left| \nabla \Big[\psi^{\,\frac{t(n-1)}{n}}\left(x_0, u_{l,\rho}/\rho\right)
\zeta^{\,\frac{\vartheta(n-1)}{n}}\Big] \right|\,dx
 \\
 &\leqslant \gamma t\int_{\Omega_{l,r(1+\sigma)}}
 \psi^{\,\frac{t(n-1)}{n}-1}\left(x_0, u_{l,\rho}/\rho\right)
 \frac{g\left(x_0, u_{l,\rho}/\rho\right)}
 {u_{l,\rho}}\,|\nabla u|\,\zeta^{\,\frac{\vartheta(n-1)}{n}}\, dx
 \\
&+\frac{\gamma\, \vartheta}{\sigma\rho}\int_{B_{r(1+\sigma)}(x_0)}
\psi^{\,\frac{t(n-1)}{n}}
\left(x_0, u_{l,\rho}/\rho\right)\zeta^{\,\frac{\vartheta(n-1)}{n}-1}\, dx.
\end{aligned}
\end{equation}
Using (${\rm g}_{3}$), \eqref{gYoungineq1}, \eqref{g<gammapsi} and (\ref{sk 2.18})
with $\tau=1-t(n-1)/n$ and $\theta=\vartheta(n-1)/n$,
we estimate the first term on the right-hand side of (\ref{sk 2.19}) as follows:
\begin{equation}\label{sk 2.20}
\begin{aligned}
&\int_{\Omega_{l,r(1+\sigma)}} \psi^{\,\frac{t(n-1)}{n}-1}
\left(x_0, u_{l,\rho}/\rho\right)
\frac{g\left(x_0, u_{l,\rho}/\rho\right)}{u_{l,\rho}}\,
|\nabla u|\,\zeta^{\,\frac{\vartheta(n-1)}{n}} dx
\\
&\leqslant \gamma e^{\lambda(\rho)}\int_{\Omega_{l,r(1+\sigma)}}
\psi^{\,\frac{t(n-1)}{n}-1}\left(x_0, u_{l,\rho}/\rho\right)
\frac{g\left(x,u_{l,\rho}/\rho\right)}{u_{l,\rho}}\,
|\nabla u|\,\zeta^{\,\frac{\vartheta(n-1)}{n}} dx
\\
&\leqslant \gamma e^{\lambda(\rho)}\int_{\Omega_{l,r(1+\sigma)}}
\psi^{\,\frac{t(n-1)}{n}-1}\left(x_0, u_{l,\rho}/\rho\right)
\frac{G\left(x, |\nabla u|\right)}{u_{l,\rho}}\,\zeta^{\,\frac{\vartheta(n-1)}{n}} dx
\\
&+\gamma\,\frac{e^{\lambda(\rho)}}{\rho} \int_{B_{r(1+\sigma)}(x_0)}
\psi^{\,\frac{t(n-1)}{n}-1}\left(x_0, u_{l,\rho}/\rho\right)
g\left(x, u_{l,\rho}/\rho\right)\,\zeta^{\,\frac{\vartheta(n-1)}{n}} dx
\\
&\leqslant \frac{\gamma\,\vartheta^{q}}{\sigma^{q}}
\left[ 1-\frac{t(n-1)}{n} \right]^{-q}
\frac{e^{2\lambda(\rho)}}{\rho}\int_{B_{r(1+\sigma)}(x_0)} \psi^{\,\frac{t(n-1)}{n}}
\left(x_0, u_{l,\rho}/\rho\right) \,\zeta^{\,\frac{\vartheta(n-1)}{n}-q}\, dx.
\end{aligned}
\end{equation}
Combining \eqref{sk 2.19}, \eqref{sk 2.20},  we arrive at
\begin{equation}\label{sk 2.21}
\begin{aligned}
\bigg(&\fint_{B_{r}(x_{0})}
\psi^{\,t}\left(x_0, u_{l,\rho}/\rho\right) dx\bigg)^{\frac{n-1}{n}}
\\
&\leqslant \frac{\gamma\,\vartheta^{q}}{\sigma^{q}}
\left( 1-\frac{t(n-1)}{n} \right)^{-q}
e^{2\lambda(\rho)}
\fint_{B_{r(1+\sigma)}(x_0)}
\psi^{\,\frac{t(n-1)}{n}} \left(x_0, u_{l,\rho}/\rho\right) dx,
\end{aligned}
\end{equation}
for $0<t<n/(n-1)$, $\vartheta\geqslant nq/(n-1)$.

Now, let $\delta_{1}\leqslant s<n/(n-1)$,
and let $j$ be a non-negative integer number such that
\begin{equation}\label{tj+1<d<tj}
s\left(\frac{n-1}{n}\right)^{j+1} \leqslant \delta_1\leqslant s\left(\frac{n-1}{n}\right)^{j}.
\end{equation}
Setting in \eqref{sk 2.21} $\vartheta=nq$, $r=r_{i}= \dfrac{\rho}{8}(2-2^{-i})$,
$r(1+\sigma)=r_{i+1}$,
$B_{i}= B_{r_{i}}(x_0)$ and $t=t_{i}= s\left(\frac{n-1}{n}\right)^{i}$ for $i=0, 1,\ldots,j+1$,
we have
$$
\begin{aligned}
\bigg(&\fint_{B_{i}}
\psi^{\,t_{i}}\left(x_0, u_{l,\rho}/\rho\right) dx\bigg)^{1/t_{i}}
\\
&\leqslant
\bigg[\gamma\, 2^{iq}
\left( 1-\frac{n-1}{n}s \right)^{-q}
e^{2\lambda(\rho)} \bigg]^{1/t_{i+1}}
\bigg(\fint_{B_{i+1}}
\psi^{\,t_{i+1}}\left(x_0, u_{l,\rho}/\rho\right) dx\bigg)^{1/t_{i+1}}.
\end{aligned}
$$
Iterating this relation and using H\"{o}lder's inequality, we obtain
the following:
\begin{equation*}
\begin{aligned}
\bigg(&\fint_{B_{\rho/8}(x_0)} \psi^{\,s}
\left(x_0, u_{l,\rho}/\rho\right) dx\bigg)^{1/s}
= \bigg(  \fint_{B_{0}}
\psi^{\,t_{0}}\left(x_0, u_{l,\rho}/\rho\right)dx \bigg)^{1/t_{0}}
\\
&\leqslant \prod\limits_{i=0}^j\bigg[ \gamma\, 2^{iq} e^{2\lambda(\rho)}
\bigg(1-\frac{n-1}{n}s \bigg)^{-q}
\bigg]^{1/t_{i+1}}\bigg(\fint_{B_{j+1}} \psi^{\,t_{j+1}}
\left(x_0, u_{l,\rho}/\rho\right) dx\bigg)^{1/t_{j+1}}
\\
&\leqslant
2^{q \sum\limits_{i=0}^{j} i/t_{i+1}}
\bigg[ \gamma e^{2\lambda(\rho)}
\bigg(1-\frac{n-1}{n}s \bigg)^{-q}
\bigg]^{\sum\limits_{i=0}^{j}1/t_{i+1}}
\bigg(\gamma \fint_{B_{\rho/4}(x_0)} \psi^{\,\delta_1}
(x_0, u_{l,\rho}/\rho) dx\bigg)^{1/\delta_1},
\end{aligned}
\end{equation*}
and by \eqref{tj+1<d<tj}, \eqref{defdelta1} and \eqref{ShSkr2.15} we have
$$
\sum\limits_{i=0}^{j}\frac{1}{t_{i+1}}
\leqslant
\frac{1}{\delta_{1}}\, \frac{n}{n-1}
\sum\limits_{i=0}^{\infty} \left(\frac{n-1}{n}\right)^{i}
=\frac{n^{2}}{\delta_{1}(n-1)},
$$
$$
\sum\limits_{i=0}^{j}\frac{i}{t_{i+1}}\leqslant
j \sum\limits_{i=0}^{j}\frac{1}{t_{i+1}}\leqslant
\frac{\gamma (\lambda(\rho)+1)}{\delta_{1}}.
$$
From this, recalling the definition of $\delta_1$
(see again \eqref{defdelta1} and \eqref{ShSkr2.15}),
we arrive at the required inequality \eqref{sk 2.17}.
This completes the proof of the lemma.
\end{proof}

Combining Lemmas \ref{lemShSkrVoit2.3} and \ref{sk.lem 2.4}, we obtain that
\begin{equation*}
\bigg(\fint_{B_{\rho/8}(x_0)}
g^{s}\left(x_0, u_{l,\rho}/\rho\right) dx\bigg)^{1/s}
\leqslant \Lambda(\gamma, 3n, \rho) g\left(x_0,
\frac{M_{l}(\rho)-M_{l}(\rho/4)+2b_0\rho}{\rho}\right),
\end{equation*}
for $0<s<n/(n-1)$, which proves Theorem \ref{thweakHarnineq}.

%%%%%%%%%%%%%%%%%%%%%%%%%%%%%%%%%%%%%%%%%%%%%%%%%%%%%%%%%%%%%%%%%%%%%%%%%%%%%%%%%%%%%%%%%%%%%%%%%%%%%%%%%%
%%%%%%%%%%%%%%%%%%%%%%%%%%%%%%%%%%%%%%%%%%%%%%%%%%%%%%%%%%%%%%%%%%%%%%%%%%%%%%%%%%%%%%%%%%%%%%%%%%%%%%%%%%
%%%%%%%%%%%%%%%%%%%%%%%%%%%%%%%%%%%%%%%%%%%%%%%%%%%%%%%%%%%%%%%%%%%%%%%%%%%%%%%%%%%%%%%%%%%%%%%%%%%%%%%%%%%

\section{Growth estimate on the gradient: proof of Theorem \ref{thgradest}}
\label{Sectgrest}

Throughout this section we assume that the hypotheses and notation of Theorem \ref{thgradest}
are in force. For definiteness, we assume that $l\geqslant M_{f}(\rho)$ in \eqref{defulrho}.
In the case where $l\leqslant m_{f}(\rho)$ in \eqref{defulrho}, the proof is completely similar.
We first prove the following inequality:
\begin{equation}\label{estgrthgrad}
\int_{\Omega_{l,\rho/8}}G(x,|\nabla u|)\,\zeta^{\,q}\,dx
\leqslant \Lambda(\gamma, 3n,\rho)M_{l}(\rho)\rho^{n-1}
g\left( x_{0},\frac{M_{l}(\rho)-M_{l}(\rho/4)+2b_{0}\rho}{\rho} \right),
\end{equation}
where %$\Omega_{l,\rho/8}=\{u>l\}\cap B_{\rho/8}(x_{0})$,
$\zeta\in C_{0}^{\infty}(B_{\rho/8}(x_{0}))$,
$0\leqslant\zeta\leqslant1$, $\zeta=1$ in $B_{\rho/16}(x_{0})$, and
$|\nabla \zeta|\leqslant 16/\rho$.
%$\Lambda(c,\beta,\rho)= \exp\Big(c\exp \big( \beta\lambda(\rho)  \big)  \Big)$
%for any $c$, $\beta\in \mathbb{R}$ and $\rho\in (0,r_{\ast})$.
For this purpose, we test \eqref{gelintidentity} by $\varphi=u_{l}\zeta^{\,q}$ and obtain
\begin{equation}\label{eq3.1}
\int_{\Omega_{l,\rho/8}} G(x, |\nabla u|)\,\zeta^{\,q}\,dx
\leqslant\gamma\, \frac{M_{l}(\rho)}{\rho}
\int_{\Omega_{l,\rho/8}} g(x, |\nabla u|)\,\zeta^{\,q-1}\,dx.
\end{equation}
The integral on the right-hand side of \eqref{eq3.1} is estimated by Young's inequality
\eqref{gYoungineq1} with $a=|\nabla u|$, $b=\zeta^{-1}$ and
\begin{equation}\label{defepsilon}
\varepsilon=\frac{\rho}{u_{l,\rho}}\,
\frac{\psi^{\tau}\left(x_{0},\frac{M_{l}(\rho)-M_{l}(\rho/4)+2b_{0}\rho}{\rho}\right)}
{\psi^{\tau}(x_{0}, u_{l,\rho}/\rho)},
\end{equation}
where the function $\psi$ is defined in \eqref{defpsi}, and
\begin{equation}\label{deftau}
0<\tau<\frac{1}{(q-1)(n-1)}.
\end{equation}
Thus, it follows from \eqref{eq3.1} that
\begin{equation}\label{estGI1I2}
\int_{\Omega_{l,\rho/8}} G(x, |\nabla u|)\,\zeta^{\,q}\,dx
\leqslant I_{1}+I_{2},
\end{equation}
where
\begin{equation}\label{defI1}
I_{1}=
\gamma\, \frac{M_{l}(\rho)}{\rho}\int_{\Omega_{l,\rho/8}}
\varepsilon\, G(x, |\nabla u|)\,\zeta^{\,q}\,dx,
\end{equation}
$$
I_{2}=
\gamma\, \frac{M_{l}(\rho)}{\rho}\int_{\Omega_{l,\rho/8}}
g\left(x, \frac{1}{\varepsilon\zeta}\right)\zeta^{\,q-1}\,dx.
$$
Let's estimate $I_{2}$ from above. Using \eqref{gqineq} with
${\rm w}=(\varepsilon\zeta)^{-1}$, ${\rm v}=\varepsilon^{-1}$
and taking into account \eqref{defepsilon}, we obtain
$$
\begin{aligned}
I_{2}
%=
%\gamma\, \frac{M_{l}(\rho)}{\rho}\int_{\Omega_{l,\rho/8}}
%g\left(x, \frac{1}{\varepsilon\zeta}\right)\zeta^{\,q-1}\,dx
\leqslant
&\gamma\, \frac{M_{l}(\rho)}{\rho}\int_{\Omega_{l,\rho/8}}
g(x,\varepsilon^{-1})\,dx
\\
&=\gamma\, \frac{M_{l}(\rho)}{\rho}\int_{\Omega_{l,\rho/8}}
g\Bigg( x, \frac{u_{l,\rho}}{\rho}\,
\frac{\psi^{\tau}(x_{0}, u_{l,\rho}/\rho)}
{\psi^{\tau}\left(x_{0},\frac{M_{l}(\rho)-M_{l}(\rho/4)+2b_{0}\rho}{\rho}\right)} \Bigg)dx.
\end{aligned}
$$
The resulting integral is estimated using condition (${\rm g}_{3}$)
and inequality \eqref{gqineq} with
$$
{\rm w}=\frac{u_{l,\rho}}{\rho}\,
\frac{\psi^{\tau}(x_{0}, u_{l,\rho}/\rho)}
{\psi^{\tau}\left(x_{0},\frac{M_{l}(\rho)-M_{l}(\rho/4)+2b_{0}\rho}{\rho}\right)},
\quad {\rm v}=\frac{u_{l,\rho}}{\rho},
$$
that yields
$$
I_{2}\leqslant \gamma\, \frac{e^{\lambda(\rho)}M_{l}(\rho)}{\rho}\int_{B_{\rho/8}}
\frac{\psi^{\tau(q-1)}(x_{0}, u_{l,\rho}/\rho)}
{\psi^{\tau(q-1)}\left(x_{0},\frac{M_{l}(\rho)-M_{l}(\rho/4)+2b_{0}\rho}{\rho}\right)}\,
g(x_{0}, u_{l,\rho}/\rho)\,dx.
$$
Hence, taking into account \eqref{g<gammapsi} and \eqref{psi<g}, we derive the estimate
$$
I_{2}\leqslant
\gamma\, \frac{e^{\lambda(\rho)}M_{l}(\rho)}{\rho}\,
g^{\tau(1-q)}\left(x_{0},\frac{M_{l}(\rho)-M_{l}(\rho/4)+2b_{0}\rho}{\rho}\right)
\int_{B_{\rho/8}} \big[g(x_{0}, u_{l,\rho}/\rho)\big]^{1+\tau(q-1)}\,dx.
$$
By \eqref{deftau} we have $1+\tau(q-1)<n/(n-1)$. Therefore,
it is possible to apply the weak Harnack inequality \eqref{weakHrnckineq} to obtain
\begin{equation}\label{estI2}
I_{2}\leqslant \Lambda(\gamma, 3n, \rho)M_{l}(\rho)\rho^{n-1}\,
g\bigg( x_{0}, \frac{M_{l}(\rho)-M_{l}(\rho/4)+2b_{0}\rho}{\rho}  \bigg).
\end{equation}
In order to estimate $I_{1}$, note that by virtue of \eqref{defepsilon} and \eqref{defI1}
we have the equality
$$
I_{1}=\gamma M_{l}(\rho)
\psi^{\tau}\bigg(x_{0},\frac{M_{l}(\rho)-M_{l}(\rho/4)+2b_{0}\rho}{\rho}\bigg)
\int_{\Omega_{l,\rho/8}} \psi^{-\tau}(x_{0}, u_{l,\rho}/\rho)
\frac{G(x,|\nabla u|)}{u_{l,\rho}}\,\zeta^{\,q}\,dx.
$$
Hence, applying successively \eqref{sk 2.18}, \eqref{psi<g}
and the weak Harnack inequality \eqref{weakHrnckineq}, we derive the inequality
\begin{equation}\label{estI1}
\begin{aligned}
I_{1}&\leqslant \gamma\,
\frac{e^{\lambda(\rho)}M_{l}(\rho)}{\rho}\,
\psi^{\tau}\left(x_{0},\frac{M_{l}(\rho)-M_{l}(\rho/4)+2b_{0}\rho}{\rho}\right)
\int_{B_{\rho/8}(x_{0})}\psi^{1-\tau}(x_{0}, u_{l,\rho}/\rho)\,dx
\\
&\leqslant \gamma\,
\frac{e^{\lambda(\rho)}M_{l}(\rho)}{\rho}\,
g^{\tau}\left(x_{0},\frac{M_{l}(\rho)-M_{l}(\rho/4)+2b_{0}\rho}{\rho}\right)
\int_{B_{\rho/8}(x_{0})}g^{1-\tau}(x_{0}, u_{l,\rho}/\rho)\,dx
\\
&\leqslant \Lambda(\gamma,3n, \rho)M_{l}(\rho)\rho^{n-1}
g\left(x_{0},\frac{M_{l}(\rho)-M_{l}(\rho/4)+2b_{0}\rho}{\rho}\right).
\end{aligned}
\end{equation}
Combining inequalities \eqref{estGI1I2}, \eqref{estI2} and \eqref{estI1},
we arrive at \eqref{estgrthgrad}.

We are now in a position to complete the proof of the inequality \eqref{eq1.7}.
Let $\eta\in C_{0}^{\infty}(B_{\rho/16}(x_{0}))$ be a function such that
$0\leqslant\eta\leqslant1$, $\eta=1$ in $B_{3\rho/64}(x_{0})$ and
$|\nabla\eta|\leqslant 64/\rho$, then
\begin{equation}\label{eq3.9}
\begin{aligned}
\int_{\Omega_{l,\rho/16}}G(x, |\nabla(u_{l,\rho}\eta)|)\,dx
&\leqslant
\int_{\Omega_{l,\rho/16}}G(x, |\nabla u|\eta+64u_{l,\rho}/\rho )\,dx
\\
&\leqslant\int_{\Omega_{l,\rho/16}}g(x, |\nabla u|+64u_{l,\rho}/\rho )\,
(|\nabla u|+64u_{l,\rho}/\rho)\,dx.
\end{aligned}
\end{equation}
In addition, due to \eqref{gqineq} the following holds on the set $\Omega_{l,\rho/16}$:
\begin{equation}\label{eq3.10}
\begin{aligned}
g(x, |\nabla u|+64u_{l,\rho}/\rho )&\leqslant
g(x, 2\max\{ |\nabla u|, 64u_{l,\rho}/\rho\} )
\\
&\leqslant c_{1}128^{q-1}g(x, \max\{ |\nabla u|, u_{l,\rho}/\rho\} )
\\
&\leqslant c_{1}128^{q-1} (g(x, |\nabla u|)+ g(x, u_{l,\rho}/\rho) ).
%\ \ \text{a.e. on} \ \Omega_{l,\rho/16}.
\end{aligned}
\end{equation}
Substituting \eqref{eq3.10} into \eqref{eq3.9}, expanding the brackets and splitting the terms
$g(x,u_{l,\rho}/\rho)|\nabla u|$ and $g(x,|\nabla u|)u_{l,\rho}/\rho$ by Young's inequality
\eqref{gYoungineq1}, we obtain
\begin{equation}\label{eq3.11}
\int_{\Omega_{l,\rho/16}}G(x, |\nabla(u_{l,\rho}\eta)|)\,dx
\leqslant
\gamma \bigg(\int_{\Omega_{l,\rho/16}}G(x,|\nabla u|)\,dx
+ \int_{\Omega_{l,\rho/16}}G(x, u_{l,\rho}/\rho)\,dx \bigg).
\end{equation}
The first integral on the right-hand side of \eqref{eq3.11}
is estimated by using \eqref{estgrthgrad}:
\begin{equation}\label{eq3.12}
\begin{aligned}
\int_{\Omega_{l,\rho/16}}G(x,|\nabla u|)\,dx
&\leqslant
\int_{\Omega_{l,\rho/8}}G(x,|\nabla u|)\,\zeta^{\,q}dx
\\
&\leqslant
\Lambda(\gamma, 3n,\rho)M_{l}(\rho)\rho^{n-1}
g\left( x_{0},\frac{M_{l}(\rho)-M_{l}(\rho/4)+2b_{0}\rho}{\rho} \right),
\end{aligned}
\end{equation}
and the second one, by using condition (${\rm g}_{3}$) and the weak Harnack inequality
\eqref{weakHrnckineq}:
\begin{equation}\label{eq3.13}
\begin{aligned}
\int_{\Omega_{l,\rho/16}}G(x, u_{l,\rho}/\rho)\,dx
\leqslant
\gamma\rho^{-1}(M_{l}(\rho)+2b_{0}\rho)e^{\lambda(\rho)}
\int_{B_{\rho/8}(x_{0})}g(x_{0},u_{l,\rho}/\rho)\,dx
\\
\leqslant
\Lambda(\gamma, 3n,\rho)(M_{l}(\rho)+2b_{0}\rho)\rho^{n-1}
g\left( x_{0},\frac{M_{l}(\rho)-M_{l}(\rho/4)+2b_{0}\rho}{\rho} \right).
\end{aligned}
\end{equation}
Collecting \eqref{eq3.11}--\eqref{eq3.13}, we arrive at \eqref{eq1.7}.
The proof is complete.

%%%%%%%%%%%%%%%%%%%%%%%%%%%%%%%%%%%%%%%%%%%%%%%%%%%%%%%%%%%%%%%%%%%%%%%%%%%%%%%%%%%%%%%%%%%%%%%%%%%%%%%%%%%%%%%%%%%%%%%%%%%%%
%%%%%%%%%%%%%%%%%%%%%%%%%%%%%%%%%%%%%%%%%%%%%%%%%%%%%%%%%%%%%%%%%%%%%%%%%%%%%%%%%%%%%%%%%%%%%%%%%%%%%%%%%%%%%%%%%%%%%%%%%%%%%
%%%%%%%%%%%%%%%%%%%%%%%%%%%%%%%%%%%%%%%%%%%%%%%%%%%%%%%%%%%%%%%%%%%%%%%%%%%%%%%%%%%%%%%%%%%%%%%%%%%%%%%%%%%%%%%%%%%%%%%%%%%%%

\bigskip

CONTACT INFORMATION

\medskip

Oleksandr V.~Hadzhy\\
Institute of Applied Mathematics and Mechanics,
National Academy of Sciences of Ukraine, Gen. Batiouk Str. 19, 84116 Sloviansk, Ukraine\\
aleksanderhadzhy@gmail.com

\medskip
Mykhailo V.~Voitovych\\Institute of Applied Mathematics and Mechanics,
National Academy of Sciences of Ukraine, Gen. Batiouk Str. 19, 84116 Sloviansk, Ukraine\\voitovichmv76@gmail.com

\end{document}